\documentclass[11pt]{article}
\usepackage{amsmath}
\usepackage{amssymb}
\usepackage{latexsym,bm}
\usepackage{amsthm}
\usepackage{graphicx}

\usepackage[top=2.4cm,bottom=2cm,left=2.4cm,right=2.4cm]{geometry}
\usepackage{algorithm}
\usepackage{algorithmic}
\usepackage{cite}
\allowdisplaybreaks
\usepackage{multirow}
\usepackage{threeparttable}
\usepackage{slashbox}

\providecommand{\U}[1]{\protect\rule{.1in}{.1in}}
\providecommand{\U}[1]{\protect\rule{.1in}{.1in}}
\providecommand{\U}[1]{\protect\rule{.1in}{.1in}}
\providecommand{\U}[1]{\protect\rule{.1in}{.1in}}

\newtheorem{theorem}{Theorem}[section]
\newtheorem{corollary}[theorem]{Corollary}
\newtheorem{lemma}{Lemma}[section]
\newtheorem{proposition}[theorem]{Proposition}
\theoremstyle{definition}
\newtheorem{definition}{Definition}[section]
\newtheorem{remark}{Remark}[section]
\numberwithin{equation}{section}
\theoremstyle{example}

\numberwithin{equation}{section}

\begin{document}

\title{Convergence analysis of a variable metric forward-backward splitting algorithm with applications}

\author{ Fuying Cui$^{a}$, Yuchao Tang$^{a}$, Chuanxi Zhu$^{a}$
\\
{\small ${^a}$ Department of Mathematics, Nanchang University,  Nanchang 330031, P.R. China }}

 \date{}

\maketitle

{}\textbf{Abstract.}\
The forward-backward splitting algorithm is a popular operator-splitting method for solving monotone inclusion of
 the sum of a maximal monotone operator and a cocoercive operator. In this paper, we present a new convergence analysis of a variable metric forward-backward splitting algorithm with extended relaxation parameters in real Hilbert spaces.
We prove that this algorithm is weakly convergent when certain weak conditions are imposed upon the relaxation parameters. Consequently, we recover the forward-backward splitting algorithm with variable step sizes. As an application, we obtain a variable metric forward-backward splitting algorithm for solving the minimization problem of the sum of two convex functions, where one of them is differentiable with a Lipschitz continuous gradient.
 Furthermore, we discuss the applications of this algorithm to the fundamental of the variational inequalities problem, constrained convex minimization problem, and split feasibility problem. Numerical experimental results on LASSO problem in statistical learning demonstrate the effectiveness of the proposed iterative algorithm.



\section{Introduction}
\label{sec-introduction}

Let $H$ be a real Hilbert space. The forward-backward splitting algorithm is a classical operator-splitting algorithm, which solves the monotone inclusion problem,
\begin{equation}\label{two-monotone-inclusion}
\textrm{find}\, x\in H,\quad \textrm{such that}\, 0\in Ax+Bx,
\end{equation}
where $A:H\rightarrow 2^{H}$ is a maximal monotone operator and $B:H\rightarrow H$ is a $\beta$-cocoercive operator, for some $\beta> 0$. The forward-backward splitting algorithm, which dates back to the original work of Lions and Mercier \cite{lionsandmercier1979}, has been studied and reported extensively in the literature, for example \cite{Chenhg1997,Tseng2000SIAM,Combettes2004Optimization,Lopez2012AAA,Zong2018}. The emergence of compressive sensing theory and large-scale optimization problems associated with signal and image processing has resulted in the forward-backward splitting algorithm receiving much attention in recent years.
A forward-backward splitting algorithm with relaxation and errors in Hilbert spaces was proposed by Combettes \cite{Combettes2004Optimization}. More precisely, let $x_{0}\in H$, set
\begin{equation}\label{relax-forward-backward}
x_{k+1}=x_{k}+\lambda_{k}(J_{\gamma_{k} A}(x_{k}-\gamma_{k}(Bx_{k}+b_{k}))+a_{k}-x_{k}),\quad k\geq 0,
\end{equation}
where $\{\gamma_{k}\}\subset (0,2\beta)$, $\{\lambda_{k}\}\subset (0,1]$, $\{a_{k}\}$ and $\{b_{k}\}$ are absolutely summable sequences in $H$. In addition, $J_{\gamma _{k}A}:=(I+\gamma_{k}A)^{-1}$ denotes the resolvent of operator $A $ with index $\gamma_{k}>0$. Combettes \cite{Combettes2004Optimization} proved the convergence of the iterative scheme (\ref{relax-forward-backward}) when certain conditions are imposed upon the parameters.  Jiao and Wang \cite{Wangfh2014JAM} generalized the iterative scheme (\ref{relax-forward-backward}) by
extending the work of Combettes \cite{Combettes2004Optimization}. They proved the convergence of (\ref{relax-forward-backward}) by requiring the parameters $\{\lambda_{k}\}$ to comply with the requirement $\{\lambda_k\} \subset \left(0,\frac{4\beta}{2\beta+\gamma_k}\right)$ when $b_k =0$. It is easy to see that $\frac{4\beta}{2\beta+\gamma_k}$ is strictly larger than one when $\{\gamma_{k}\}\subset (0,2\beta)$. Further, Combettes and Yamada \cite{Combettes2015JMAA} improved the range of the relaxation parameters $\{\lambda_k\}$ in (\ref{relax-forward-backward}) to $(0,\frac{4\beta -\gamma_k}{2\beta})$. After a simple calculation, we know that $\frac{4\beta -\gamma_k}{2\beta} > \frac{4\beta}{2\beta+\gamma_k}$. Therefore, the range of $\{\lambda_k\}$ in the work of Combettes and Yamada \cite{Combettes2015JMAA} is larger than that of Jiao and Wang \cite{Wangfh2014JAM}.

  In the case when $\gamma_k = \gamma$ and $a_{k}=b_{k}=0$, the iterative scheme (\ref{relax-forward-backward}) is reduced to the forward-backward splitting algorithm with a constant step size \cite{bauschkebook2017},
\begin{equation}\label{relax-forward-backward-constant}
x_{k+1}=x_{k}+ \lambda_{k}(J_{\gamma A}(x_{k}-\gamma Bx_{k})-x_{k}),\quad k\geq 0,
\end{equation}
where $\gamma \in (0,2\beta)$ and $\{\lambda_k\} \subset (0, \frac{4 \beta-\gamma}{2 \beta})$.
Bauschke and Combettes \cite{bauschkebook2017} obtained the convergence of the iterative algorithm (\ref{relax-forward-backward-constant}) by adopting the Krasnoseki\u{i}-Mann iteration for computing the fixed points of nonexpansive operators.
 The forward-backward splitting algorithm with constant step size (\ref{relax-forward-backward-constant}) is usually considered to be stationary, whereas the forward-backward splitting algorithm with variable step sizes (\ref{relax-forward-backward}) is referred to as non-stationary.

It is worth mentioning that by letting $\lambda_{k}=1$, then $(\ref{relax-forward-backward-constant})$ reduces to the classical forward-backward splitting algorithm. More precisely, the iterative sequence $\{x_k\}$ is defined by
\begin{equation}\label{classical-forward-backward}
x_{k+1}=J_{\gamma A}(x_{k}-\gamma Bx_{k}), \quad  k\geq 0.
\end{equation}
In the context of convex optimization, the forward-backward splitting algorithm is equivalent to the so-called proximal gradient algorithm (PGA) applied to solve the following convex minimization problem,
\begin{equation}\label{sum-two-convex}
\min_{x\in H} f(x)+g(x),
\end{equation}
where $f:H\rightarrow R$ is convex, differentiable with an $L$-Lipschitz continuous gradient for some $L>0$ and $g:H\rightarrow (-\infty,+\infty]$ is a proper, lower semicontinuous, convex function. The convex optimization problem (\ref{sum-two-convex}) has found widespread application in signal and image processing, for example \cite{beck2009,beckandteboulle2009,Caijf2010SIAM}. As a consequence of \cite{Combettes2004Optimization}, Combettes and Wajs \cite{combettes2005} employed the forward-backward splitting algorithm (\ref{relax-forward-backward}) to solve the minimization problem (\ref{sum-two-convex}). The obtained iterative algorithm is defined as
\begin{equation}\label{forward-backward}
x_{k+1}=x_{k}+\lambda_{k}(prox_{\gamma_ {k} g}(x_{k}-\gamma_{k}(\nabla f(x_{k})+b_{k}))+a_{k}-x_{k}),\quad k\geq 0,
\end{equation}
where $\{\gamma_{k}\}\subset (0,2/L)$, $\{\lambda_{k}\}\subset (0,1]$, and $\{a_{k}\}$, $\{b_{k}\}$ are absolutely summable sequences in $H$. $prox_{\gamma g}$ denotes the proximity operator of $g$ with index $
\gamma >0$.
In addition, Combettes and Wajs \cite{combettes2005} presented applications of this algorithm to many concrete convex optimization problems. This iterative algorithm (\ref{forward-backward}) was subsequently improved by Combettes and Yamada \cite{Combettes2015JMAA} who extended the range of the relaxation parameters $\{\lambda_k\}$.

Inspired by solving large-scale convex optimization problems arising in image processing, machine learning, and economic management, many efficient primal-dual splitting algorithms have been proposed for structured monotone inclusions involving maximal monotone operators and single-valued Lipschitz or cocoercive monotone operators, for example \cite{combettes2012,vu2013ACM}. Although these monotone inclusions are more complicated than the monotone inclusion problem (\ref{two-monotone-inclusion}), they can be transformed into the form of this problem in a suitable product space. Therefore, it is natural to consider using the forward-backward splitting algorithm (e.g., (\ref{relax-forward-backward}) or (\ref{relax-forward-backward-constant})) to solve the equivalent monotone inclusion problem. Because the backward steps cannot be decomposed, direct use of the forward-backward splitting algorithm often fails to obtain a completely splitting algorithm.
Many researchers attempted to overcome this difficulty by investigating variable metric operator splitting algorithms. The use of a suitable variable metric enables the implicit step of backward splitting to be easily decomposed. For example, the primal-dual hybrid gradient algorithm \cite{Esser2010} (also known as the primal-dual of the Chambolle-Pock algorithm \cite{chambolleandpock2011}) is equivalent to the variable metric proximal point algorithm \cite{Burke1998SIAM,Parente2008SIAM}. We refer the readers to a subsequent paper \cite{he2012} for more details. V\~{u} \cite{Vu2013NFAO} proposed a variable metric extension of the forward-backward-forward splitting algorithm \cite{Tseng2000SIAM} for solving monotone inclusion of the sum of
a maximal monotone operator and a monotone Lipschitzian operator in Hilbert spaces. Liang \cite{Liangj2016Diss} proposed a variable metric multi-step inertial operator-splitting algorithm for solving the monotone inclusion problem (\ref{two-monotone-inclusion}). Bonettini et al. \cite{Bonettini2016SIAMJSC} developed a scaled inertial forward-backward splitting algorithm for solving (\ref{two-monotone-inclusion}) in the context of convex minimization. Neither of the respective algorithms in the work by Liang \cite{Liangj2016Diss} and Bonettini et al. \cite{Bonettini2016SIAMJSC} was compatible with the relaxation strategy.
The variable metric forward-backward splitting algorithm was originally studied in finite-dimensional Hilbert spaces \cite{Chenhg1997,Lotito2009JCA}; however, the methods in these studies either had to be strongly monotone to study the convergence rate or they did not make use of the cocoercive property of $B$ in (\ref{two-monotone-inclusion}). For infinite-dimensional Hilbert spaces, Combettes and V\~{u} \cite{Combettes2014Optimization} proposed a variable metric forward-backward splitting algorithm to solve (\ref{two-monotone-inclusion}) and analyzed its weak and strong convergence. This algorithm is defined as follows. Let $x_{0}\in H $, and set
\begin{equation}\label{variable-forward-backward}
\left\{
\begin{aligned}
y_{k}   & =  x_{k}-\gamma_{k}U_{k}(Bx_{k}+b_{k}), \\
x_{k+1} & =  x_{k}+\lambda_{k}(J_ {\gamma_{k}U_{k}A}(y_{k})+a_{k}-x_{k}).
\end{aligned}
\right.
\end{equation}
where $\{U_{k}\}\subset \mathcal{P}_{\alpha}(H)$, $\{\lambda_{k}\}\subset (0,1]$, $\{\gamma_{k}\}\subset (0,2\beta)$, $ \{a_{k}\}$ and $\{b_{k}\} $ are absolutely summable sequences in $H$. This algorithm (\ref{variable-forward-backward}) includes a variable metric, variable step sizes, relaxation parameter, and errors. It includes nearly all of the forward-backward type of splitting algorithms mentioned above. For example, by
letting $U_{k}=I $ in (\ref{variable-forward-backward}), it is reduced to (\ref{relax-forward-backward}).
The relaxation parameters $\{\lambda_{k}\}$ in (\ref{relax-forward-backward}) are observed to be strictly larger than that based on the work of Combettes and Yamada \cite{Combettes2015JMAA}. While preparing this manuscript, we discovered that in Chapter 5 of the dissertation \cite{Simoes2017D}, Sim\~{o}es generalized the variable metric forward-backward splitting algorithm by replacing the relaxation parameters $\{\lambda_k\}$ in (\ref{variable-forward-backward}) with self-adjoint, strong positive linear operators. However, this approach still requires the maximum eigenvalue of the operators to be smaller than one.

The purpose of this paper is to introduce a new convergence analysis for the variable metric forward-backward splitting algorithm (\ref{variable-forward-backward}) with an extended range of relaxation parameters.
We prove the weak convergence of the variable metric forward-backward splitting algorithm by setting the relaxation parameter $\{\lambda_{k}\}$ larger than one in real Hilbert spaces. To achieve this goal, we make full use of the averaged and firmly nonexpansive property of operators $J_{\gamma_k U_kA}(I-\gamma_k U_kB)$ and $J_{\gamma_k U_kA}$, where $\lambda_k >0$ and $U_k\in \mathcal{P}_{\alpha}(H)$. In contrast, existing solutions mainly rely on $J_{\gamma_k U_k A}$ being firmly nonexpansive. Consequently, we obtain the convergence of the forward-backward splitting algorithm with variable step sizes. Moreover, we impose a slightly weak condition on the relaxation parameters to ensure the convergence of this algorithm. The results we obtained complement and extend those of Combettes and Yamada \cite{Combettes2015JMAA}.
 As an application, we obtain the variable metric forward-backward splitting algorithm for solving the minimization problem (\ref{sum-two-convex}).
 We also present the application of this algorithm to the variational inequalities problem, constrained convex minimization problem, and split feasibility problem. To the best of our knowledge, the iterative algorithms we obtained are the most general ones for solving these problems. Finally, we conduct numerical experiments on LASSO problem to validate the effectiveness of the proposed iterative algorithm.

The remainder of this paper is organized as follows. Section $2$ reviews selected notations and lemmas on monotone operator theory and presents some technical lemmas. In Section $3$,
we prove the main convergence results of the variable metric forward-backward splitting algorithm with relaxation in real Hilbert spaces. Consequently, we obtain several corollaries of some special cases. Section $4$ presents our use of the proposed iterative algorithm to solve three typical optimization problems include the variational inequalities problem, constrained convex minimization problem, and split feasibility problem. In Section $5$, we present preliminary numerical results on LASSO problem to illustrate the performance of the proposed iterative algorithm.
Finally, we provide our conclusions.

\section{Preliminaries}
\label{sec:pre}
In this section, we recall selected concepts and lemmas that are commonly used in the context of
convex analysis and monotone operator theory. Most of them can be found in \cite{bauschkebook2011,bauschkebook2017}. Throughout this paper, let $H$ be a real Hilbert space. The inner product and the associated norm of Hilbert space $H$ are denoted by $\langle\cdot,\cdot\rangle$ and $\|\cdot\|$, respectively. $I$ denotes the identity operator and the symbols $\rightharpoonup$ and $\rightarrow$ denote weak and strong convergence.

We first recall selected basic notations and definitions. Let $A:H\rightarrow 2^H$ be a set-valued operator. We denote its domain, range, graph, and zeros by dom $A= \{ x\in H| Ax \neq \emptyset \}$, ran $A = \{ u\in H | (\exists x\in H) u\in Ax\}$,  gra $A = \{ (x,u)\in H\times H | u\in Ax \}$, and zer $A = \{x\in H  | 0\in Ax\}$, respectively.

\begin{definition}
Let $A:H\rightarrow 2^H$ be a set-valued operator. $A$ is said to be monotone, if
$$
\langle x-y, u-v \rangle \geq 0, \quad \forall (x,u), (y,v)\in \textrm{gra } A.
$$
Moreover, $A$ is said to be maximal monotone, if its graph is not strictly contained in the graph of any other monotone operator on $H$.
\end{definition}
A well-known example of a maximal monotone operator is the subgradient mapping of a proper, lower semicontinuous convex function $f:H \rightarrow (-\infty, +\infty]$ defined by
$$
\partial f: H\rightarrow 2^H : x \mapsto \{ u\in H | f(y) \geq f(x) + \langle u, y-x\rangle, \forall y\in H \}.
$$

\begin{definition}
Let $A:H\rightarrow 2^H$ be a maximal monotone operator. The resolvent operator of $A$ with index $\lambda >0$ is defined as
$$
J_{\lambda A} = (I+\lambda A)^{-1}.
$$
\end{definition}

According to the Minty theorem, the resolvent operator $J_{\lambda A}$ is defined everywhere on Hilbert space $H$, and $J_{\lambda A}$ is firmly nonexpansive.

Let us recall the definition of the proximity operator, which was first introduced by Moreau \cite{Moreau1962}. Let $f\in \Gamma_{0}(H)$, where $\Gamma_{0}(H)$ denotes the set of all proper lower semicontinuous convex functions $f:H\rightarrow (-\infty, +\infty]$. The proximity operator of $f$ with index $\lambda >0$ is defined by
$$
prox_{\lambda f} : H\rightarrow H: x \mapsto \arg\min_{y\in H} \frac{1}{2}\| y-x \|^2 + \lambda f(y).
$$
In fact, the resolvent operator of the subdifferential operator of any $f\in \Gamma_{0}(H)$ with index $\lambda >0$ is the proximal operator of $f$ with index $\lambda >0$, that is
$$
prox_{\lambda f} = (I+\lambda \partial f)^{-1}.
$$
Therefore, the proximity operators have the same property as the resolvent operators.

\begin{definition}
Let $B:H\rightarrow H$ be a single-valued operator. Let $\beta >0$, then $B$ is said to be $\beta$-cocoercive, if
$$
\langle x-y, Bx-By \rangle \geq \beta \| Bx-By \|^2, \quad \forall x,y\in H.
$$
\end{definition}

The $\beta$-cocoercive operator is also known as a $\beta$-inverse strongly monotone operator ($\beta$-ism), for example \cite{Byrne2004}. It is easy to see from the above definition that a $\beta$-cocoercive operator is $\frac{1}{\beta}$-Lipschitz continuous, i.e., $\|Bx-By\| \leq \frac{1}{\beta}\|x-y\| $.

Next, we recall the definitions of nonexpansive and related mappings. These mappings often appear in the convergence analysis of optimization algorithms.

\begin{definition}
Let $C$ be a nonempty subset of $H$. Let $T:C\rightarrow H$, then

\noindent \emph{(i)} $T$ is considered to be nonexpansive, if
$$
\|Tx-Ty\| \leq \|x-y\|, \quad \forall x,y\in C.
$$

\noindent \emph{(ii)} $T$ is considered to be firmly nonexpansive, if
$$
\|Tx-Ty\|^2 \leq \|x-y\|^2 - \| (I-T)x- (I-T)y \|^2, \quad \forall x,y\in C.
$$

\noindent \emph{(iii)} $T$ is referred to as $\alpha$-averaged, $\alpha\in (0,1)$, if there exists a nonexpansive mapping $S$ such that $T = (1-\alpha)I + \alpha S$.

\end{definition}

It follows immediately that a firmly nonexpansive mapping is a nonexpansive mapping and an $\alpha$-averaged mapping is also nonexpansive.

We denote by $Fix(T)$ the set of fixed pints of a mapping $T$, that is $Fix(T) = \{x\in H | x = Tx\}$.

\begin{lemma}(Demiclosedness Principle)
Let $C$ be a nonempty subset of $H$. Let $T:C\rightarrow H$ be a nonexpansive mapping with $Fix(T)\neq \emptyset$. If $\{x_k\}$ is a sequence in $C$ that converges weakly to $x$ and if $\{(I-T)x_k\}$ converges strongly to $y$, then $(I-T)x=y$; in particular, if $y=0$, then $x\in Fix(T)$.

\end{lemma}

The following proposition provides some equivalent definitions of the firmly nonexpansive mappings. This proposition can be found in Proposition 4.2 of \cite{bauschkebook2011}.

\begin{proposition}\label{prop1}
Let $C$ be a nonempty subset of $H$. Let $T:C\rightarrow H$, then the following are equivalent

\noindent \emph{(i)} $T$ is firmly nonexpansive;

\noindent \emph{(ii)} $I-T$ is firmly nonexpansive;

\noindent \emph{(iii)} $2T-I$ is nonexpansive;

\noindent \emph{(iv)} $\langle x-y, Tx-Ty \rangle \geq \|Tx-Ty\|^2, \quad \forall x,y\in C$;

\end{proposition}

From Proposition \ref{prop1} (iii) and (iv), we know that if $T$ is firmly nonexpansive, then $T$ is $\frac{1}{2}$-averaged, and a $1$-cocoercive operator is firmly nonexpansive.

The following proposition is taken from Proposition 4.25 of \cite{bauschkebook2011}.

\begin{proposition}\label{prop2}
Let $C$ be a nonempty subset of $H$. Let $T:C\rightarrow H$, then $T$ is $\alpha$-averaged if and only if
$$
\|Tx-Ty\|^2 \leq \|x-y\|^2 - \frac{1-\alpha}{\alpha}\|(I-T)x-(I-T)y\|, \quad \forall x,y\in C.
$$.

\end{proposition}

The following lemma provides a relation between an operator $T$ with its complement $I-T$.

\begin{lemma}\label{lemma-complement}
Let $C$ be a nonempty subset of $H$. Let $T:C\rightarrow H$, then

\noindent \emph{(i)} $T$ is nonexpansive if and only if the complement $I-T$ is $\frac{1}{2}$-cocoercive;

\noindent \emph{(ii)} $T$ is $\alpha$-averaged if and only if the complement $I-T$ is $\frac{1}{2\alpha}$-cocoercive.

\end{lemma}

We refer interested readers to \cite{bauschkebook2011} for further properties of nonexpansive, firmly nonexpansive, and $\alpha$-averaged nonlinear mappings.

We recall the results of the composition of two averaged operators. The following lemma first appeared in \cite{oguraandyamada2002} after which it was extended to a finite family of composition-averaged operators \cite{Combettes2015JMAA}.

\begin{lemma}\label{composition-averaged}
Let $C$ be a nonempty subset of $H$. Let $T_1 : C\rightarrow H$ is $\alpha_1$-averaged and $T_2 :C\rightarrow H$ is $\alpha_2$-averaged. Then
$$
T := T_1 T_2 \textrm{ is } \frac{\alpha_1 +\alpha_2 - 2\alpha_1 \alpha_2}{1-\alpha_1 \alpha_2}-\textrm{averaged}.
$$

\end{lemma}

\begin{remark}
(i)\ It is worth mentioning that two other results of the combination of averaged operators were reported. From Proposition 4.32 of \cite{bauschkebook2011}, $T := T_1 T_2$ is $\overline{\alpha} = \frac{2}{1+\frac{1}{\max(\alpha_1,\alpha_2)}}$-averaged. From Byrne \cite{Byrne2004}, $T := T_1 T_2$ is $\widehat{\alpha} = \alpha_1 + \alpha_2 -\alpha_1 \alpha_2$-averaged. It is not difficult to verify that $\frac{\alpha_1 +\alpha_2 - 2\alpha_1 \alpha_2}{1-\alpha_1 \alpha_2}$ is smaller than the other two constants $\overline{\alpha}$ and $\widehat{\alpha}$.

(ii)\ The constant $\widehat{\alpha}$ is used in \cite{Wangfh2014JAM} to show the upper bound of the relaxation parameter $\lambda_k$ such that $\lambda_k < \frac{1}{\widehat{\alpha}}$.

\end{remark}

We employ the following previously used notation \cite{Combettes2014Optimization}. Let $\mathcal{B}(H,G)$ be the spaces of bounded linear operators from Hilbert space $H$ to Hilbert space $G$. The norm of $L\in \mathcal{B}(H,G)$ is defined as $\|L\| = \sup_{x\in H}\frac{\|Lx\|}{\|x\|}$. We set $\mathcal{B}(H)= \mathcal{B}(H,H)$ and $\mathbb{S}(H)=\{L\in \mathcal{B}(H)| L=L^{*}\}$, where $L^{*}$ denotes the adjoint of $L$. The Loewner partial ordering on $S(H)$ is defined by, for any $U,V\in S(H)$,
\begin{equation}
U \succeq V \Leftrightarrow  \langle Ux,x\rangle \geq \langle Vx,x\rangle, \forall x\in H . \nonumber
\end{equation}
Let $\alpha \in [0, +\infty)$, set
\begin{equation}
\mathcal{P}_{\alpha}(H)=\{U\in S(H)|U\succeq \alpha I\}. \nonumber
\end{equation}
We denote $\sqrt{U}$ as the square root of $U\in \mathcal{P}_{\alpha}(H)$. Moreover, for every $U\in \mathcal{P}_{\alpha}(H)$, we define a semi-scalar product and a semi-norm (a scalar product and a norm , if $\alpha>0$ by
\begin{equation}
(\forall x\in H) (\forall y\in H)\quad \langle x,y \rangle_{U}= \langle Ux,y\rangle \quad  and \quad  \left\|x\right\|_{U}=\sqrt{\langle Ux,x\rangle}. \nonumber
\end{equation}

We borrow the following results on monotone operators in a variable metric setting from Combettes’s work \cite{Combettes2014Optimization}.

\begin{lemma}\label{variable-metric-maximal-monotone}
Let $A:H\rightarrow 2^{H}$ be maximal monotone, let $\alpha\in (0,+\infty)$, let $U\in \mathcal{P}_{\alpha}(H)$
and let $H_{U^{-1}}$ be the real Hilbert space with the scalar product $\langle x,y \rangle_{U^{-1}}=\langle U^{-1}x,y\rangle , \forall x,y\in H$.
Then the following hold:

\noindent (i) $UA:H\rightarrow 2^{H}$ is maximal monotone;

\noindent (ii) $J_{UA}:H\rightarrow 2^{H}$ is $1$-cocoercive, i.e., firmly nonexpansive. More precisely,
\begin{equation}
\|J_{UA} x-J_{UA} y\|^{2}_{U^{-1}}\leq \|x-y\|^{2}_{U^{-1}}-\|(I-J_{UA})x-(I-J_{UA})y\|^{2}_{U^{-1}}, \quad \forall x,y\in H.
\end{equation}
\noindent (iii) $J_{UA}=(U^{-1}+A)^{-1}\circ U^{-1}$
\end{lemma}

Let $U\in \mathcal{P}_{\alpha}(H)$ for some $\alpha >0$. The proximity operator of $f\in \Gamma_{0}(H)$ relative to the metric induced by $U$ is defined by
$$
prox_{f}^{U} :H\rightarrow H : x \mapsto \arg\min_{y\in H} \left ( \frac{1}{2}\|x-y\|_{U}^{2} + f(y)    \right).
$$
We have $prox_{f}^{U} = J_{U^{-f}\partial f}$ and we can write $prox_{f}^{I}=prox_{f}$.

We make full use of the following lemmas to obtain the weak convergence of the considered iterative sequence. Both of the two lemmas were previously reported \cite{Combettes2013NA}. In the following, we denote by $\ell_{+}^{1}(\mathbb{N})$ the set of summable sequences in $[0,+\infty)$, where $\mathbb{N}$ is a set of nonnegative integer numbers.

\begin{lemma}\label{conergence-lemma1}
Let $\alpha \in (0,+\infty)$, and let $\{W_{k}\} $ be in $\mathcal{P}_{\alpha}(H)$, let $C$ be a nonempty subset of $H$, and let $\{x_{k}\}$ be a sequence in $H$ such that
\begin{equation}\label{lemma-inequality}
\left\| x_{k+1}-z\right\|_{W_{k+1}}\leq (1+\eta_{k})\left\| x_{k}-z\right\|_{W_{k}}+\epsilon_{k}, \forall z\in C,
\end{equation}
where $\{\eta_{n}\}\subset \ell_{+}^{1}(\mathbb{N})$ and $\{\epsilon_{k}\}\subset \ell_{+}^{1}(\mathbb{N})$.
Then $\{x_{k}\}$ is bounded and, for every $z \in C$, $(\left\| x_{k}-z\right\|_{W_{k}})$ converges.
\end{lemma}

\begin{lemma}\label{convergence-lemma2}
Let $\alpha \in (0,+\infty)$, and let $\{W_{k}\} $ and $W$ be in $\mathcal{P}_{\alpha}(H)$ such that $W_k \rightarrow W$ pointwise as $k\rightarrow +\infty$, as is the case when
$$
sup_{k\in N} \left\|W_{k}\right\| < +\infty \textrm{ and } (\exists \{\eta_k\}\subset \ell_{+}^{1}(\mathbb{N})) (1+\eta_k)W_k \succeq W_{k+1}.
$$
Let $C$ be a nonempty subset of $H$, and let $\{x_{k}\}$ be a sequence in $H$ such that (\ref{lemma-inequality}) is satisfied.
Then $\{x_{k}\}$ converges weakly to a point in $C$ if and only if every weak sequential cluster point of $\{x_{k}\}$ is in C.
\end{lemma}

The following lemma can be found in Corollary 2.14 in the book by Bauschke and Combettes \cite{bauschkebook2011}.
\begin{lemma}\label{lemma-equation}
Let $x\in H, y \in H$, and $\alpha \in R$. Then
\begin{equation}\label{lemma-equality}
\left\|\alpha x + (1-\alpha)y \right\|^{2}=\alpha \left\|x\right\|^{2}+(1-\alpha) \left\|x\right\|^{2}-\alpha (1-\alpha)\left\|x-y\right\|^{2}
\end{equation}
\end{lemma}

\section{Variable metric forward-backward splitting algorithm}

In this section, we study the convergence of the variable metric forward-backward splitting algorithm. First, we prove the following useful lemmas.

\begin{lemma}\label{our-lemma1}
Let $B:H\rightarrow H $ be a $\beta$-cocoercive operator. Let $\alpha>0$, and let $U\in \mathcal{P}_{\alpha}(H)$.
Let $H_{U^{-1}}$ be a real Hilbert space with the scalar product $\langle x,y\rangle_{U^{-1}}=\langle U^{-1}x,y\rangle, \forall x,y\in H$.
Then $I-\gamma UB$ is a $\frac{\gamma \|U\|}{2\beta}$-averaged operator on $H_{U^{-1}}$, for any $\gamma\in(0,\frac{2\beta}{\|U\|})$.
\end{lemma}

\begin{proof}
Let $x,y\in H$. Because $B$ is $\beta$-cocoercive, we have
\begin{align}\label{our-lem-eq1}
\langle UBx-UBy,x-y\rangle_{U^{-1}}
& = \langle Bx-By,x-y\rangle \nonumber \\
& \geq \beta \|Bx-By\|^{2}.
\end{align}
On the other hand, we obtain
\begin{equation}\label{our-lem-eq2}
\|UBx-UBy\|^{2}_{U^{-1}}\leq \|U\|\cdot\|Bx-By\|^{2}.
\end{equation}
From (\ref{our-lem-eq1}) and (\ref{our-lem-eq2}), we obtain
\begin{equation}
\langle UBx-UBy,x-y\rangle_{U^{-1}}\geq \frac{\beta}{\|U\|}\cdot\|UBx-UBy\|^{2}_{U^{-1}},
\end{equation}
which means that $UB$ is $\frac{\beta}{\|U\|}$-cocoercive on $H_{U^{-1}}$.
Then $\gamma UBx$ is $\frac{\beta}{\gamma\|U\|}$-cocoercive. By Lemma \ref{lemma-complement} (ii), $I-\gamma UB$ is $\frac{\gamma\|U\|}{2\beta}$-averaged  operator on
$H_{U^{-1}}$.
\end{proof}

\begin{lemma}\label{our-lemma2}
Let $A:H\rightarrow 2^{H}$ be maximal monotone. Let $\alpha\in (0,+\infty)$, and let $U\in \mathcal{P}_{\alpha}(H)$. Let $H_{U^{-1}}$ be a real
Hilbert space with the scalar product $\langle x,y\rangle_{U^{-1}}= \langle U^{-1}x,y\rangle,\forall x,y\in H$. Let $B:H\rightarrow H$ be a $\beta$-cocoercive operator.
Then, for any $\gamma\in(0,\frac{2\beta}{\|U\|}), J_{\gamma UA}(I-\gamma UB)$ is $\frac{2\beta}{4\beta-\gamma\|U\|}$-averaged on $H_{U^{-1}}$.
\end{lemma}

\begin{proof}
Because $A$ is maximal monotone, then for any $\gamma>0$, $\gamma UA$ is maximal monotone. According to Lemma \ref{variable-metric-maximal-monotone} (ii), $J_{\gamma UA}$ is $1$-cocoercive on $H_{U^{-1}}$.
Then $J_{\gamma UA}$ is $\frac{1}{2}$-averaged. Lemma \ref{our-lemma1} determines that $I-\gamma UB$ is $\frac{\gamma\|U\|}{2\beta}$-averaged. Therefore, we apply Lemma \ref{composition-averaged},
from which we know that $J_{\gamma UA}(I-\gamma UB)$ is
\begin{equation}
\frac{\alpha_{1}+\alpha_{2}-2\alpha_{1}\alpha_{2}}{1-\alpha_{1}\alpha_{2}}=\frac{\frac{1}{2}+\frac{\gamma\|U\|}{2\beta}
-\frac{\gamma\|U\|}{2\beta}}{1-\frac{1}{2}\cdot\frac{\gamma\|U\|}{2\beta}}=\frac{2\beta}{4\beta-\gamma\|U\|},
\end{equation}
which is the averaged operator.
\end{proof}

\begin{lemma}\label{our-lemma3}
Let $H$ be a real Hilbert space. Let $A:H \rightarrow 2^{H}$ be a maximal monotone operator. Let $B:H \rightarrow H$ be a $\beta$-cocoercive operator, for some $\beta > 0$. Suppose that $\Omega:=$zer$(A+B)\neq \emptyset$. Let $\gamma_k >0$, $\alpha > 0$, and $\{U_{k}\}\subset \mathcal{P}_{\alpha}(H)$.
Then the following are equivalent:

\noindent \emph{(i)}  $x^{*} \in zer(A+B)$.\\
\noindent \emph{(ii)} $x^{*} = J_ {\gamma_{k}U_{k}A}(I - \gamma_{k}U_{k}B)(x^{*})$, for any $\gamma_{k}> 0$.\\
\noindent \emph{(iii)} $x^{*} = (\frac{U_{k}^{-1}+\gamma_{k} A}{\alpha})^{-1} \circ (\frac{U_{k}^{-1}-\gamma_{k} B}{\alpha})x^{*}$.
\end{lemma}

\begin{proof}
(i)$\Leftrightarrow$(ii) Let $x^{*} \in zer(A+B)$, then we have
\begin{align}
&\qquad 0\in \gamma_{k}Ax^{*}+ \gamma_{k}Bx^{*} \nonumber \\
&\Leftrightarrow 0\in \gamma_{k}U_{k}Ax^{*}+ \gamma_{k}U_{k}Bx^{*} \nonumber \\
&\Leftrightarrow x^*- \gamma_{k}U_{k}Bx^{*} \in x^{*}+ \gamma_{k}U_{k}Ax^{*} \nonumber \\
&\Leftrightarrow x^{*}=(I+\gamma_{k}U_{k}A)^{-1}(x^*-\gamma_{k}U_{k}Bx^{*})\nonumber \\
&\Leftrightarrow x^* = J_ {\gamma_{k}U_{k}A}(I - \gamma_{k}U_{k}B)(x^*)\nonumber
\end{align}
(ii)$\Leftrightarrow(iii)$ Let $x^{*}= J_{\gamma_{k}U_{k}A}(I-\gamma_{k}U_{k}B)x^{*}$, then
\begin{align}
&\qquad x^{*}-\gamma_{k}U_{k}Bx^{*}\in x^{*}+\gamma_{k}U_{k}Ax^{*}\nonumber \nonumber \\
&\Leftrightarrow U_{k}^{-1}x^{*}-\gamma_{k}Bx^{*}\in U_{k}^{-1}x^{*}+\gamma_{k}Ax^{*} \nonumber \\
&\Leftrightarrow (\frac{U_{k}^{-1}-\gamma_{k}B}{\alpha})x^{*} \in (\frac{U_{k}^{-1}+\gamma_{k}A}{\alpha})x^{*} \nonumber \\
&\Leftrightarrow x^{*} = (\frac{U_{k}^{-1}+\gamma_{k} A}{\alpha})^{-1} \circ (\frac{U_{k}^{-1}-\gamma_{k} B}{\alpha})x^{*} \nonumber.
\end{align}
\end{proof}

\begin{lemma}\label{our-lemma4}
Let $H$ be a real Hilbert space. Let $A:H \rightarrow 2^{H}$ be a maximal monotone operator. Let $B:H \rightarrow H$ be a $\beta$-cocoercive operator, for some $\beta > 0$. Let $r>0$ and $s>0$, and let $U, V\in \mathcal{P}_{\alpha}(H)$. Define a variable metric forward-backward operator $T_{r U} := J_{rUA}(I-rUB)$. Then, for any $x\in H$, we have
$$
\|T_{rU}x - T_{sV}x  \| \leq \frac{1}{\lambda_{min}(U^{-1})} \left\| \left (U^{-1} - \frac{r}{s}V^{-1} \right) (x- T_{sV}x) \right\|,
$$
where $\lambda_{min}(U^{-1})$ represents the minimum eigenvalue of $U^{-1}$.
\end{lemma}

\begin{proof}
Let $x\in H$, in which case we have
\begin{align*}
& \frac{U^{-1}x-U^{-1}T_{rU}x}{r} - Bx \in A T_{rU}x, \\
& \frac{V^{-1}x-V^{-1}T_{sV}x}{s} - Bx \in A T_{sV}x.
\end{align*}
It follows from the monotonicity of operator $A$ that
$$
\left\langle T_{rU}x - T_{sV}x,  \frac{U^{-1}x-U^{-1}T_{rU}x}{r} - \frac{V^{-1}x-V^{-1}T_{sV}x}{s}   \right\rangle \geq 0.
$$
Then
$$
\| T_{rU}x - T_{sV}x \|_{U^{-1}}^{2} \leq r \left\langle T_{rU}x - T_{sV}x, \left( \frac{U^{-1}}{r} - \frac{V^{-1}}{s} \right)(x-T_{sV}x)   \right \rangle.
$$
Because of the Cauchy-Schwarz inequality and the fact that $\lambda_{min}(U^{-1})\|x\|^2 \leq \|x\|_{U^{-1}}^{2}$, for any $x\in H$, we obtain
$$
\|T_{rU}x - T_{sV}x  \| \leq \frac{1}{\lambda_{min}(U^{-1})} \left\| \left (U^{-1} - \frac{r}{s}V^{-1} \right) (x- T_{sV}x) \right\|.
$$

\end{proof}

We are ready to state our main theorems and present their convergence analysis.

\begin{theorem}\label{main-theorem}
Let $H$ be a real Hilbert space. Let $A:H \rightarrow 2^{H}$ be maximal monotone. Let $B:H\rightarrow H$ be $\beta$-cocoercive, for some $\beta> 0$. Suppose that $\Omega:=zer(A+B)\neq \emptyset$. Let $\alpha> 0$, $\{\eta_{k}\}\in \ell_{+}^{1}(\mathbb{N})$, and $\{U_{k}\}\subset \mathcal{P}_{\alpha}(H)$ such that
\begin{equation}\label{u-condition}
\mu = \sup_{k\in \mathbb{N}} \left\|U_{k}\right\| < +\infty  \quad \textrm{and} \quad   ( 1+ \eta_{k}) U_{k+1} \succeq U_{k}, \quad \forall k \in \mathbb{N}.
\end{equation}
Let $\{\gamma_{k}\} \subset (0, \frac{2\beta}{\|U_k\|})$, and $\{\lambda_k\} \subset (0,\frac{1}{\alpha_{k}})$, where $\alpha_{k}=\frac{2 \beta}{4 \beta-\gamma_{k} \|U_k\|}$.
Let $\{a_{k}\}$ and $\{b_{k}\}$ be two sequences in $H$ such that $\sum_{k=0}^{+\infty}\lambda_k \|a_k\|<+\infty$ and $\sum_{k=0}^{+\infty}\lambda_k \|b_k\|<+\infty$. Let $x_{0}\in H$, and set
\begin{equation}\label{variable-forward-backward2}
\left\{
\begin{aligned}
y_{k}   & =  x_{k}-\gamma_{k}U_{k}(Bx_{k}+b_{k}), \\
x_{k+1} & =  x_{k}+\lambda_{k}(J_ {\gamma_{k}U_{k}A}(y_{k})+a_{k}-x_{k}),
\end{aligned}
\right.
\end{equation}

Then, we have

\noindent \emph{(i)} For any $x^{*}\in \Omega$, $\lim_{k\rightarrow +\infty}\|x_{k}-x^{*}\|_{U_{k}^{-1}}$ exists;

Suppose that $0< \underline{\lambda}\leq \lambda_{k} \leq \frac{1}{\alpha_{k}}-\tau$, where $\tau \in (0,\frac{1}{\alpha_{k}}-\underline{\lambda})$, then

\noindent \emph{(ii)} $\lim_{k\rightarrow +\infty} \| x_{k} - J_{\gamma_k U_k A}(x_k - \gamma_k U_k Bx_k) \| =0 $;

Suppose that $0< \underline{\gamma}\leq \gamma_{k}$, then

\noindent \emph{(iii)}  $\{x_{k}\}$ converges weakly to a point in $\Omega$.

Further, suppose that $ \gamma_{k} \leq \frac{2\beta-\epsilon}{\mu}$, where $\epsilon \in (0,2\beta-\mu\underline{\gamma})$. Then

\noindent \emph{(iv)} $Bx_k \rightarrow Bx^{*}$ as $k\rightarrow +\infty$, where $x^{*}\in \Omega$.

\end{theorem}

\begin{proof}
According to condition (\ref{u-condition}), we have
\begin{equation}
\|U_{k}^{-1}\| \leq \frac{1}{\alpha},\quad U_{k}^{-1}\in \mathcal{P}_{\frac{1}{\mu}}(H), \quad \textrm { and } (1+\eta_k)U_{k}^{-1} \succeq U_{k+1}^{-1}.
\end{equation}
Hence,
\begin{equation}\label{u-inequality}
( 1+ \eta_{k}) \left\|x\right\|_{U_{k}^{-1}}^{2} \geq \left\|x\right\|_{U_{k+1}^{-1}}^{2},  \quad \forall x\in H.
\end{equation}

For the sake of convenience, let
\begin{equation}\label{eq3-1}
\overline{x}_{k+1} = x_{k}+\lambda_{k}(J_ {\gamma_{k}U_{k}A}(x_{k}-\gamma_{k}U_{k}Bx_{k})-x^{k}).
\end{equation}
Then, iterative scheme (\ref{variable-forward-backward2}) can be rewritten as
\begin{equation}\label{eq3-2}
x_{k+1} = \overline{x}_{k+1} + \lambda_k e_k,
\end{equation}
where $e_k = J_ {\gamma_{k}U_{k}A}(y_k) - J_ {\gamma_{k}U_{k}A}(x_{k}-\gamma_{k}U_{k}Bx_{k}) + a_k$ such that $\sum_{k=0}^{+\infty}\lambda_k \|e_k\|<+\infty$. In fact, because $J_{\gamma_k U_k A}$ is nonexpansive on $H_{U_{k}^{-1}}$, we have
\begin{align}\label{error-estimation}
\lambda_k \|e_k\| & \leq \sqrt{\mu} \lambda_k \|e_k\|_{U_{k}^{-1}} \nonumber \\
& \leq \sqrt{\mu} \lambda_k \| y_k - (x_{k}-\gamma_{k}U_{k}Bx_{k}) \|_{U_{k}^{-1}} + \sqrt{\mu} \lambda_k \|a_k\|_{U_{k}^{-1}} \nonumber \\
& \leq \mu \gamma_k \lambda_k \|b_k\| +  \sqrt{\frac{1}{\alpha}} \lambda_k \|a_k\| \nonumber \\
& \leq \mu \frac{2\beta}{\alpha} \lambda_k \|b_k\| +  \sqrt{\frac{1}{\alpha}} \lambda_k \|a_k\|.
\end{align}
Notice that $\sum_{k=0}^{+\infty}\lambda_k \|a_k\|<+\infty$ and $\sum_{k=0}^{+\infty}\lambda_k \|b_k\|<+\infty$, (\ref{error-estimation}) implies that $\sum_{k=0}^{+\infty}\lambda_k \|e_k\|<+\infty$.

From Lemma \ref{our-lemma2}, we know that $J_ {\gamma_{k}U_{k}A}(I-\gamma_{k}U_{k}B)$ is $\frac{2\beta}{4\beta - \gamma_k \|U_k\|}$-averaged. Let $\alpha_k = \frac{2\beta}{4\beta - \gamma_k \|U_k\|}$, then there exist nonexpansive mappings $R_k$ such that $J_ {\gamma_{k}U_{k}A}(I-\gamma_{k}U_{k}B) = (1-\alpha_k)I + \alpha_k R_k$. Consequently, the iterative sequence $\{\overline{x}_{k+1}\}$ in (\ref{eq3-1}) is equivalent to
\begin{align}\label{eq3-3}
\overline{x}_{k+1} & = (1-\lambda_k)x_k + \lambda_k ( (1-\alpha_k)x_k + \alpha_k R_k x_k  ) \nonumber \\
& = (1-\lambda_k \alpha_k)x_k + \lambda_k \alpha_k R_k x_k.
\end{align}

(i)  Let $x^{*}\in zer(A+B)$, according to Lemma \ref{our-lemma3}, $x^{*} = J_ {\gamma_{k}U_{k}A}(I - \gamma_{k}U_{k}B)(x^{*})$. Then $x^* = R_k x^*$.
From (\ref{u-inequality}), (\ref{eq3-2}), and (\ref{eq3-3}), we obtain
\begin{align}\label{eq3-4}
\|x_{k+1} - x^* \|_{U_{k+1}^{-1}} & \leq \sqrt{(1+\eta_k)}\| x_{k+1} -x^* \|_{U_{k}^{-1}} \nonumber \\
&  \leq \sqrt{(1+\eta_k)} ( \| \overline{x}_{k+1} - x^* \|_{U_{k}^{-1}} + \lambda_k \| e_k \|_{U_{k}^{-1}}  ) \nonumber \\
& \leq \sqrt{(1+\eta_k)} \| (1-\lambda_k \alpha_k )(x_k - x^*) + \lambda_k \alpha_k (R_k x_k - x^*)   \|_{U_{k}^{-1}} \nonumber \\
& \quad + \sqrt{(1+\eta_k)} \sqrt{\frac{1}{\alpha}}\lambda_k \|e_k\| \nonumber \\
& \leq (1+\eta_k) \|x_k - x^*\|_{U_{k}^{-1}} + \epsilon_k,
\end{align}
where $\epsilon_k = \sqrt{(1+\eta_k)}\sqrt{\frac{1}{\alpha}}\lambda_k \|e_k\|$. Because $\sum_{k=0}^{+\infty}\lambda_k \|e_k\| < +\infty$ and $\sum_{k=0}^{+\infty}\|\eta_k\| < +\infty$, then $\sum_{k=0}^{\infty}\|\epsilon_k\|<+\infty$. On the basis of Lemma \ref{conergence-lemma1}, we conclude that $\lim_{k\rightarrow +\infty}\|x_k - x^*\|_{U_{k}^{-1}}$ exists. Moreover, $\{\|x^k - x^*\|\}$ is bounded. Let $M>0$ such that $\sup_{k\geq 0}\|x^k - x^*\| \leq M$.

(ii) With the help of the inequality $\|x+y\|^2 \leq \|x\|^2 + 2\langle y, x+y\rangle$, $\forall x,y\in H$. We obtain
\begin{align}\label{eq3-5}
\|x_{k+1} - x^* \|_{U_{k+1}^{-1}}^2 & \leq (1+\eta_k) \| x_{k+1} -x^* \|_{U_{k}^{-1}}^2 \nonumber \\
& = (1+\eta_k) \| \overline{x}_{k+1} -x^* + \lambda_k e_k  \|_{U_{k}^{-1}}^{2} \nonumber \\
& = (1+\eta_k) ( \| \overline{x}_{k+1} -x^* \|_{U_{k}^{-1}}^{2} + 2\lambda_k \langle e_k, x_{k+1}-x^*  \rangle_{U_{k}^{-1}} ) \nonumber \\
& \leq (1+\eta_k) \| \overline{x}_{k+1} -x^* \|_{U_{k}^{-1}}^{2} + 2(1+\eta_k)M \|U_{k}^{-1}\| \lambda_k\|e_k\|.
\end{align}

From Lemma \ref{lemma-equation} and (\ref{eq3-1}) we derive that
\begin{align}\label{eq3-6}
\left\|\overline{x}_{k+1}-x^{*}\right\|_{U_{k}^{-1}}^{2} &= \left\|(1-\lambda_{k})(x_{k}-x^{*})+\lambda_{k}(J_{\gamma_{k}U_{k}A}(x_{k}
-\gamma_{k}U_{k}Bx_{k})-x^{*})\right\|_{U_{k}^{-1}}^{2} \nonumber  \\
&=(1-\lambda_{k})\left\|x_{k}-x^{*}\right\|_{U_{k}^{-1}}^{2}
+\lambda_{k}\left\|(J_{\gamma_{k}U_{k}A}(x_{k}-\gamma_{k}U_{k}Bx_{k})-x^{*})\right\| _{U_{k}^{-1}}^{2}  \nonumber \\
&\quad-\lambda_{k}(1-\lambda_{k})\left\|x_{k}-J_{\gamma_{k}U_{k}A}(x_{k}-\gamma_{k}U_{k}Bx_{k})\right\|_{U_{k}^{-1}}^{2}.
\end{align}
Because $J_{\gamma_{k}U_{k}A}(I-\gamma_{k}U_{k}B)$ is $\alpha_{k}$-averaged, it follows from Proposition \ref{prop2} that
\begin{equation}\label{eq3-7}
\left\|J_{\gamma_{k}U_{k}A}(x_{k}-\gamma_{k}U_{k}Bx_{k})-x^{*}\right\|_{U_{k}^{-1}}^{2}\leq\left\|x_{k}-x^{*}\right\|_{U_{k}^{-1}}^{2}-
\frac{1-\alpha_{k}}{\alpha_{k}}\left\|x_{k}-J_{\gamma_{k}U_{k}A}(x_{k}-\gamma_{k}U_{k}Bx_{k})\right\|_{U_{k}^{-1}}^{2}.
\end{equation}
Substituting (\ref{eq3-7}) into (\ref{eq3-6}) yields,
\begin{equation}\label{eq3-8}
\left\|\bar{x}_{k+1}-x^{*}\right\|_{U_{k}^{-1}}^{2} \leq \left\|x_{k}-x^{*}\right\|_{U_{k}^{-1}}^{2}-\lambda_{k}(\frac{1}{\alpha_{k}}-\lambda_{k})
\left\|x_{k}-J_{\gamma_{k}U_{k}A}(x_{k}-\gamma_{k}U_{k}Bx_{k})\right\|_{U_{k}^{-1}}^{2}.
\end{equation}
Combining (\ref{eq3-8}) with (\ref{eq3-5}), we obtain
\begin{align}\label{eq3-9}
\|x_{k+1} - x^* \|_{U_{k+1}^{-1}}^2 & \leq (1+\eta_k) \left\|x_{k}-x^{*}\right\|_{U_{k}^{-1}}^2 + 2(1+\eta_k)M \|U_{k}^{-1}\| \lambda_k \|e_k\| \nonumber \\
& \quad - (1+\eta_k)\lambda_{k}(\frac{1}{\alpha_{k}}-\lambda_{k})
\left\|x_{k}-J_{\gamma_{k}U_{k}A}(x_{k}-\gamma_{k}U_{k}Bx_{k})\right\|_{U_{k}^{-1}}^{2},
\end{align}
which implies  that
\begin{align}\label{eq3-10}
& \quad \lambda_{k} (\frac{1}{\alpha_{k}}-\lambda_{k})\left\|x_{k}- J_{\gamma_ {k}U_{k}A}(x_{k}-\gamma_{k}U_{k}Bx_{k})\right\|_{U_{k}^{-1}}^{2} \nonumber \\
& \leq (1+\eta_k)\left\|x_{k}-x^{*}\right\|_{U_{k}^{-1}}^{2}-\left\|x_{k+1}-x^{*}\right\|_{U_{k+1}^{-1}}^{2} + 2(1+\eta_k)M \|U_{k}^{-1}\|  \lambda_k \|e_k\|.
\end{align}
Observe that $\lim_{k\rightarrow +\infty}\left\|x_{k}-x^{*}\right\|_{U_{k}^{-1}}$ exists and $\sum_{k=0}^{+\infty}\lambda_k \|e_k\| < +\infty$. Then by letting $k\rightarrow +\infty$ in the above inequality and considering the condition on $\{\lambda_k\}$,  we obtain
\begin{equation}\label{eq3-11}
\lim_{k\to +\infty}\left\|x_{k}- J_{\gamma_ {k}U_{k}A}(x_{k}-\gamma_{k}U_{k}Bx_{k})\right\|_{U_{k}^{-1}}=0.
\end{equation}
Because the two norms $\|\cdot\|_{U_{k}^{-1}}$ and $\|\cdot\|$ defined on the Hilbert spaces $H$ are equivalent, it follows from (\ref{eq3-11}) that
\begin{equation}\label{eq3-112}
\lim_{k\to +\infty}\left\|x_{k}- J_{\gamma_ {k}U_{k}A}(x_{k}-\gamma_{k}U_{k}Bx_{k})\right\|=0.
\end{equation}

(iii) In this part, we prove that the sequence $\{x_k\}$ converges weakly to a point in $\Omega$. In fact, let $\bar{x}$ be a weak sequential cluster point of $ \{x_{k}\}$,  then there exists a subsequence $\{x_{k_{n}}\}\subset\{x_{k}\}$ such that
$x_{k_{n}}\rightharpoonup \bar{x} $. Because $\{\gamma_k\} \subset (\underline{\gamma},\frac{2\beta}{\|U_k\|}) \subset (\underline{\gamma},\frac{2\beta}{\alpha})$ is bounded, there exists a subsequence of $\{\gamma_k\}$ converges to $\gamma \in (\underline{\gamma},\frac{2\beta}{\alpha})$. Without loss of generality, we may assume that $\gamma_{k_n}\rightarrow \gamma$. According to condition (\ref{u-condition}), it follows from Lemma \ref{convergence-lemma2} that there exists $U^{-1}\in \mathcal{P}_{\frac{1}{\mu}}(H)$ such that $U_{k}^{-1}\rightarrow U^{-1}$ pointwise.

With the help of Lemma \ref{our-lemma4}, we make the following estimation,
\begin{align}\label{asympto-estimation}
& \quad \| x_{k_n} - J_{\gamma UA}(x_{k_n} - \gamma U Bx_{k_n} )   \| \nonumber \\
& \leq \| x_{k_n} - J_{\gamma_{k_n} U_{k_n}A}(x_{k_n} - \gamma_{k_n} U_{k_n} Bx_{k_n} )   \|   \nonumber \\
& \quad + \| J_{\gamma_{k_n} U_{k_n}A}(x_{k_n} - \gamma_{k_n} U_{k_n} Bx_{k_n} ) - J_{\gamma UA}(x_{k_n} - \gamma U Bx_{k_n} )   \| \nonumber \\
& \leq \| x_{k_n} - J_{\gamma_{k_n} U_{k_n}A}(x_{k_n} - \gamma_{k_n} U_{k_n} Bx_{k_n} )   \| \nonumber \\
& \quad + \frac{1}{\lambda_{min}(U_{k}^{-1})} \| \left( U_{k_n}^{-1} - \frac{\gamma_{k_{n}}}{\gamma}U^{-1}   \right)(x_{k_n}-J_{\gamma UA}(x_{k_n} - \gamma U Bx_{k_n} ))  \| \nonumber \\
& \leq \| x_{k_n} - J_{\gamma_{k_n} U_{k_n}A}(x_{k_n} - \gamma_{k_n} U_{k_n} Bx_{k_n} )   \| \nonumber \\
& \quad + \frac{\mu}{\gamma} \| (U_{k_n}^{-1}\gamma - U_{k_n}^{-1}\gamma_{k_n})(x_{k_n}-J_{\gamma UA}(x_{k_n} - \gamma U Bx_{k_n} ))  \| \nonumber \\
& \quad + \frac{\mu}{\gamma} \| ( U_{k_n}^{-1}\gamma_{k_n} - U^{-1}\gamma_{k_n})(x_{k_n}-J_{\gamma UA}(x_{k_n} - \gamma U Bx_{k_n} ))  \| \nonumber \\
& \leq \| x_{k_n} - J_{\gamma_{k_n} U_{k_n}A}(x_{k_n} - \gamma_{k_n} U_{k_n} Bx_{k_n} )   \| \nonumber \\
& \quad + \frac{\mu}{\gamma \alpha}  | \gamma - \gamma_{k_n}| \|x_{k_n}-J_{\gamma UA}(x_{k_n} - \gamma U Bx_{k_n} ) \| \nonumber \\
& \quad + \frac{\mu}{\gamma} \frac{2\beta}{\alpha} \| ( U_{k_n}^{-1} - U^{-1})(x_{k_n}-J_{\gamma UA}(x_{k_n} - \gamma U Bx_{k_n} ))  \| .
\end{align}
Because $\{\|x_{k_n}-J_{\gamma UA}(x_{k_n} - \gamma U Bx_{k_n} ) \|\}$ is bounded,  it follows from the conditions above, and we can conclude from (\ref{asympto-estimation})  that
\begin{equation}
\| x_{k_n} - J_{\gamma UA}(x_{k_n} - \gamma U Bx_{k_n} )   \| \rightarrow 0 \textrm{ as } k_n \rightarrow +\infty.
\end{equation}
As $J_{\gamma UA}(I - \gamma U B )$ is nonexpansive, based on the demiclosedness property of nonexpansive mapping, we deduce that $\bar{x} = J_{\gamma UA}(\bar{x} - \gamma U B\bar{x} )$, which means that $\bar{x} \in \textrm{ zer } (A+B)$. Because $\bar{x}$ is arbitrary, together with conclusion (i), we can conclude from Lemma \ref{convergence-lemma2} that $\{x_{k}\}$ converges weakly to a point in zer$(A+B)$.

(iv) On the other hand, as $J_{\gamma_ {k}U_{k}A}$ is firmly nonexpansive, it follows that we have
\begin{align}\label{eq3-12}
& \quad \left\|J_{\gamma_{k}U_{k}A}(x_{k}-\gamma_{k}U_{k}Bx_{k})-x^{*}\right\|_{U_{k}^{-1}}^{2} \nonumber \\
& \leq \left\|x_{k}-\gamma_{k}U_{k}Bx_{k}-(x^{*}-\gamma_{k}U_{k}Bx^{*})\right\|_{U_{k}^{-1}}^{2} \nonumber \\
&\quad - \left\|(I-J_{\gamma_ {k}U_{k}A})(x_{k}-\gamma_{k}U_{k}Bx_{k})-(I-J_{\gamma_ {k}U_{k}A})(x^{*}-\gamma_{k}U_{k}Bx^{*})\right\|_{U_{k}^{-1}}^{2} \nonumber \\
&=\left\|x_{k}-x^{*}-(\gamma_{k}U_{k}Bx_{k}-\gamma_{k}U_{k}Bx^{*})\right\|_{U_{k}^{-1}}^{2} \nonumber \\
& \quad-\left\|x_{k}-J_{\gamma_ {k}U_{k}A}(x_{k}-\gamma_{k}U_{k}Bx_{k})-(\gamma_ {k}U_{k}Bx_{k}-\gamma_{k}U_{k}Bx^{*})\right\|_{U_{k}^{-1}}^{2} \nonumber \\
&=\left\|x_{k}-x^{*}\right\|_{U_{k}^{-1}}^{2}-2 \langle x_{k}-x^{*},\gamma_{k}U_{k}Bx_{k}-\gamma_{k}U_{k}Bx^{*}\rangle_{U_{k}^{-1}} \nonumber \\
&\quad +\left\|\gamma_{k}U_{k}Bx_{k}-\gamma_{k}U_{k}Bx^{*}\right\|_{U_{k}^{-1}}^{2}\nonumber \\
& \quad-\left\|x_{k}-J_{\gamma_ {k}U_{k}A}(x_{k}-\gamma_{k}U_{k}Bx_{k})-(\gamma_ {k}U_{k}Bx_{k}-\gamma_{k}U_{k}Bx^{*})\right\|_{U_{k}^{-1}}^2.
\end{align}
Because $B$ is $\beta$-cocoercive, we have that
\begin{equation}\label{eq3-13}
\langle x_{k}-x^{*},\gamma_{k}U_{k}Bx_{k}-\gamma_{k}U_{k}Bx^{*}\rangle_{U_{k}^{-1}}\geq \gamma_{k}\beta\left\|Bx_{k}-Bx^{*}\right\|^{2}.
\end{equation}
In addition, we have
\begin{align}\label{eq3-14}
\left\|\gamma_{k}U_{k}Bx_{k}-\gamma_{k}U_{k}Bx^{*}\right\|_{U_{k}^{-1}}^{2} &\leq \gamma_{k}^{2}\left\|U_{k}\right\| \left\|Bx_{k}-Bx^{*}\right\|^{2} \nonumber \\
 &\leq \mu \gamma_{k}^{2} \left\|Bx_{k}-Bx^{*}\right\|^{2}.
\end{align}
Substituting (\ref{eq3-13}) and (\ref{eq3-14}) into (\ref{eq3-12}), we obtain
\begin{align}\label{eq3-15}
& \quad \left\|J_{\gamma_ {k}U_{k}A}(x_{k}-\gamma_{k}U_{k}Bx_{k})-x^{*}\right\|_{U_{k}^{-1}}^{2} \nonumber \\
 &\leq \left\| x_{k}-x^{*} \right\|_{U_{k}^{-1}}^{2}
-\gamma_{k}(2\beta-\gamma_{k}\mu)\left\| Bx_{k}-Bx^{*} \right\|^{2} \nonumber \\
& -\left\|( x_{k}-J_{\gamma_{k}U_{k}A}(x_{k}-\gamma_{k}U_{k}Bx_{k}))-(\gamma_{k}U_{k}Bx_{k}-\gamma_{k}U_{k}Bx^{*})\right\|_{U_{k}^{-1}}^{2}.
\end{align}

The combination of (\ref{eq3-15}) with (\ref{eq3-6}) yields
\begin{align}\label{eq3-16}
\left\|\overline{x}_{k+1}-x^{*}\right\|_{U_{k}^{-1}}^{2} &\leq \left\|x_{k}-x^{*}\right\|_{U_{k}^{-1}}^{2}
-\lambda_{k}\gamma_{k}(2\beta-\gamma_{k}\mu)\left\| Bx_{k}-Bx^{*} \right\|^{2} \nonumber\\
\quad & -\lambda_{k}\left\|x_{k}- J_{\gamma_{k}U_{k}A}(x_{k}-\gamma_{k}U_{k}Bx_{k})-(\gamma_{k}U_{k}Bx_{k}-\gamma_{k}U_{k}Bx^{*})\right\|_{U_{k}^{-1}}^{2} \nonumber\\
\quad & -\lambda_{k}(1-\lambda_{k})\left\|x_{k}- J_{\gamma_{k}U_{k}A}(x_{k}-\gamma_{k}U_{k}Bx_{k})\right\|_{U_{k}^{-1}}^{2}.
\end{align}

Further, on the basis of (\ref{eq3-16}) and (\ref{eq3-5}), we obtain
\begin{align}\label{eq3-17}
\left\|x_{k+1}-x^{*}\right\|_{U_{k+1}^{-1}}^{2} &\leq (1+\eta_k)\left\|x_{k}-x^{*}\right\|_{U_{k}^{-1}}^{2}
- (1+\eta_k)\lambda_{k}\gamma_{k}(2\beta-\gamma_{k}\mu)\left\| Bx_{k}-Bx^{*} \right\|^{2} \nonumber\\
\quad & - (1+\eta_k) \lambda_{k}\left\|x_{k}- J_{\gamma_{k}U_{k}A}(x_{k}-\gamma_{k}U_{k}Bx_{k})-(\gamma_{k}U_{k}Bx_{k}-\gamma_{k}U_{k}Bx^{*})\right\|_{U_{k}^{-1}}^{2} \nonumber\\
\quad & - (1+\eta_k) \lambda_{k}(1-\lambda_{k})\left\|x_{k}- J_{\gamma_{k}U_{k}A}(x_{k}-\gamma_{k}U_{k}Bx_{k})\right\|_{U_{k}^{-1}}^{2} \nonumber \\
\quad & + 2\lambda_k(1+\eta_k)M \|U_{k}^{-1}\| \|e_k\|,
\end{align}
which implies that
\begin{align}\label{eq3-18}
\lambda_{k}\gamma_{k}(2\beta-\gamma_{k}\mu)\left\| Bx_{k}-Bx^{*} \right\|^{2} & \leq (1+\eta_k)\left\|x_{k}-x^{*}\right\|_{U_{k}^{-1}}^{2} - \left\|x_{k+1}-x^{*}\right\|_{U_{k+1}^{-1}}^{2} \nonumber \\
& \quad -(1+\eta_k) \lambda_{k}(1-\lambda_{k})\left\|x_{k}- J_{\gamma_{k}U_{k}A}(x_{k}-\gamma_{k}U_{k}Bx_{k})\right\|_{U_{k}^{-1}}^{2} \nonumber \\
 & \quad + 2\lambda_k(1+\eta_k)M \|U_{k}^{-1}\| \|e_k\|.
\end{align}

By the conditions on $\{\gamma_k\}$ and $\{\lambda_k\}$, and together with conclusions (i), (ii) and the fact that $\sum_{k=0}^{+\infty}\lambda_k \|e_k\|<+\infty$, letting $k\rightarrow +\infty$ in the above inequality, we obtain
\begin{equation}\label{eq3-19}
Bx_{k}\rightarrow Bx^{*} \textrm{ as } k\rightarrow +\infty.
\end{equation}
This completes the proof.

\end{proof}

\begin{remark}
Because the upper bound of the relaxation parameter $\{\lambda_k\}$ in Theorem \ref{main-theorem} is governed by the averaged constant of the variable metric forward-backward operator, Theorem \ref{main-theorem} provides a larger selection of
 the relaxation parameter and errors than Theorem 4.1 of Combettes \cite{Combettes2014Optimization}.
\end{remark}

\begin{remark}

If we assume that $\lambda_k \in (\underline{\lambda},1]$, then we reaffirm the conclusion that $\sum_{k=0}^{+\infty}\|Bx_{k}-Bx^*\|^2<+\infty$ as in Theorem 4.1 of the paper by Combettes \cite{Combettes2014Optimization}. In fact,
from inequality (\ref{eq3-18}), we have
\begin{align}
\underline{\lambda}\underline{\gamma}\epsilon \left\| Bx_{k}-Bx^{*} \right\|^{2}
& \leq \lambda_{k}\gamma_{k}(2\beta-\gamma_{k}\mu)\left\| Bx_{k}-Bx^{*} \right\|^{2} \nonumber \\
&  \leq (1+\eta_k)\left\|x_{k}-x^{*}\right\|_{U_{k}^{-1}}^{2} - \left\|x_{k+1}-x^{*}\right\|_{U_{k+1}^{-1}}^{2} \nonumber \\
 & \quad + 2\lambda_k(1+\eta_k)M \|U_{k}^{-1}\| \|e_k\|. \nonumber
\end{align}
By summing the above inequality from zero to infinity, we have
\begin{align}
& \quad \underline{\lambda}\underline{\gamma}\epsilon \sum_{k=0}^{+\infty}\left\| Bx_{k}-Bx^{*} \right\|^{2} \nonumber \\
 & \leq \left\|x_{0}-x^{*}\right\|_{U_{0}^{-1}}^{2}  +\sum_{k=0}^{+\infty}\eta_k \sup_{k\geq 0}{\left\|x_{k}-x^{*}\right\|_{U_{k}^{-1}}^{2}}  \nonumber \\
& \quad + \sum_{k=0}^{+\infty}2\lambda_k(1+\eta_k)M \|U_{k}^{-1}\| \|e_k\|, \nonumber
\end{align}
which implies that $\sum_{k=0}^{+\infty}\left\| Bx_{k}-Bx^{*} \right\|^{2} < +\infty$.

\end{remark}

\begin{remark}
In view of Theorem \ref{main-theorem} (iii), the iterative sequence generated by (\ref{variable-forward-backward2}) converges weakly to a point in $\Omega$. The strong convergence of $\{x_k\}$ requires $x_k \rightarrow x^{*}$, $x^{*}\in \Omega$. Similar to Theorem 4.1 of Combettes \cite{Combettes2014Optimization}, we need to assume that one of the following conditions holds.

(i) $\liminf_{k\rightarrow +\infty}d_{\Omega}(x_k) = 0$;

(ii) $A$ or $B$ is demiregular at every point in $\Omega$;

(iii) int$\Omega \neq \emptyset$ and there exists $\{v_k\}\in \ell_{+}^{1}(\mathbb{N})$ such that $(1+v_k)U_k \succeq U_{k+1}$.

Because the proof is the same as that of Combettes \cite{Combettes2014Optimization}, we omit it here.
\end{remark}

Next, we impose a slightly weaker condition on the iterative parameter $\lambda_k$ than in Theorem \ref{main-theorem} to ensure the weak convergence of the iterative sequence $\{x_k\}$.

\begin{theorem}\label{main-theorem2}

Let $H$ be a real Hilbert space. Let $A:H \rightarrow 2^{H}$ be maximal monotone. Let $B:H\rightarrow H$ be $\beta$-cocoercive, for some $\beta> 0$. Suppose that $\Omega:=zer(A+B)\neq \emptyset$. Let $\alpha> 0$, $\{\eta_{k}\}\in \ell_{+}^{1}(\mathbb{N})$ and $\{U_{k}\}\in \mathcal{P}_{\alpha}(H)$ such that
\begin{equation}\label{}
\mu = \sup_{k\in N} \left\|U_{k}\right\| < +\infty  \quad \textrm{and} \quad   ( 1+ \eta_{k}) U_{k+1} \succeq U_{k}, \quad \forall k \in \mathbb{N}.
\end{equation}
Let the iterative sequence $\{x_k\}$ be defined by (\ref{variable-forward-backward2}). Then, we have

\noindent \emph{(i)} For any $x^{*}\in \Omega$, $\lim_{k\rightarrow +\infty}\|x_{k}-x^{*}\|_{U_{k}^{-1}}$ exists;

Suppose that

\noindent \emph{(a)} $\sum_{k=0}^{+\infty}\lambda_k (\frac{1}{\alpha_k}-\lambda_k) = +\infty$, where $\alpha_{k}=\frac{2 \beta}{4 \beta-\gamma_{k} \|U_k\|}$;

\noindent \emph{(b)} $0< \underline{\gamma}\leq \gamma_{k} \leq \frac{2\beta-\epsilon}{\mu}$, where $\epsilon \in (0,2\beta-\mu\underline{\gamma})$;

\noindent \emph{(c)} $\sum_{k=0}^{+\infty}|  \gamma_{k+1} - \gamma_{k} | < +\infty$, $\sum_{k=0}^{+\infty}|  \gamma_{k+1}\|U_{k+1}\| - \gamma_{k}\|U_{k}\| | < +\infty$, and $\sum_{k=0}^{+\infty}\|U_{k}^{-1}x-U_{k+1}^{-1}x\|<+\infty$, for any $x\in H$.

Then,

\noindent \emph{(ii)} $\lim_{k\rightarrow +\infty} \| x_{k} - J_{\gamma_k U_k A}(x_k - \gamma_k U_k Bx_k) \| =0 $;

\noindent \emph{(iii)}  $\{x_{k}\}$ converges weakly to a point in $\Omega$;

Further, suppose that $\lambda_k \geq \underline{\lambda}>0$. Then

\noindent \emph{(iv)} $Bx_k \rightarrow Bx^{*}$ as $k\rightarrow +\infty$, where $x^{*}\in \Omega$.

\end{theorem}

\begin{proof}
(i) Let $x^{*}\in \Omega$, it follows from the same proof of Theorem \ref{main-theorem} (i) and we know that $\lim_{k\rightarrow +\infty}\|x_k - x^{*}\|_{U_{k}^{-1}}$ exists. Then, $\{\|x_k - x^{*}\|\}$ is bounded. Let $M:= \sup_{k\geq 0}\|x_k - x^{*}\|$.

(ii) From (\ref{eq3-10}), we obtain
\begin{align}\label{eq3-21}
& \quad \sum_{k=0}^{+\infty}\lambda_{k} (\frac{1}{\alpha_{k}}-\lambda_{k})\left\|x_{k}- J_{\gamma_ {k}U_{k}A}(x_{k}-\gamma_{k}U_{k}Bx_{k})\right\|_{U_{k}^{-1}}^{2} \nonumber \\
& \leq \left\|x_{0}-x^{*}\right\|_{U_{0}^{-1}}^{2} + \frac{1}{\alpha} M^2 \sum_{k=0}^{+\infty}\eta_k + 2 \frac{1}{\alpha} \sum_{k=0}^{+\infty}(1+\eta_k)M  \lambda_k\|e_k\|.
\end{align}
Because $\sum_{k=0}^{+\infty}\eta_k < +\infty$ and $\sum_{k=0}^{+\infty} \lambda_k \|e_k\| < +\infty$, then
\begin{equation}\label{eq3-22}
\sum_{k=0}^{+\infty}\lambda_{k} (\frac{1}{\alpha_{k}}-\lambda_{k})\left\|x_{k}- J_{\gamma_ {k}U_{k}A}(x_{k}-\gamma_{k}U_{k}Bx_{k})\right\|_{U_{k}^{-1}}^{2} < +\infty.
\end{equation}

Let $T_k = J_{\gamma_k U_k A}(I - \gamma_k U_k B)$. By condition (a), (\ref{eq3-22}) implies that $\liminf_{k\rightarrow +\infty}\left\|x_{k}- T_k x_{k}\right\|_{U_{k}^{-1}} =0$.
Consequently, $\lim\inf_{k\rightarrow +\infty}\left\|x_{k}- T_k x_{k}\right\| =0$. Because $T_k$ is $\alpha_k$-averaged, where $\alpha_k = \frac{2\beta}{4\beta - \gamma_k \|U_k\|}$, there exists nonexpansive mappings $R_k$ on $H_{U_{k}^{-1}}$ such that
$T_k = (1-\alpha_{k})I + \alpha_{k} R_k$. Then, $\lim\inf_{k\rightarrow +\infty}\|x_k - R_k x_k\|_{U_{k}^{-1}}=0$. Next, we prove that $\lim_{k\rightarrow +\infty}\|x_k - R_k x_k\|=0$.

Using formulation (\ref{eq3-2}) and the fact that $R_{k+1}$ is nonexpansive on $H_{{U_{k+1}^{-1}}}$, we have
\begin{align}\label{eq3-23}
\|x_{k+1} - R_{k+1}x_{k+1}\|_{U_{k+1}^{-1}} & \overset{(\ref{eq3-2})}{=} \| \overline{x}_{k+1} - R_{k+1}x_{k+1} + \lambda_k e_k  \|_{U_{k+1}^{-1}} \nonumber \\
& \leq \| \overline{x}_{k+1} - R_{k+1}x_{k+1}\|_{U_{k+1}^{-1}} + \lambda_k \|e_k  \|_{U_{k+1}^{-1}} \nonumber \\
& = \| (1-\lambda_k \alpha_{k})x_k + \lambda_k \alpha_{k} R_k x_k - R_{k+1}x_{k+1}   \|_{U_{k+1}^{-1}}    + \lambda_k \|e_k  \|_{U_{k+1}^{-1}} \nonumber \\
& = \| (1-\lambda_k \alpha_{k})(x_k - R_{k}x_k) +  R_k x_k - R_{k+1}x_{k+1}   \|_{U_{k+1}^{-1}}    + \lambda_k \|e_k  \|_{U_{k+1}^{-1}} \nonumber \\
& \leq (1-\lambda_k \alpha_{k}) \|x_k - R_{k}x_k \|_{U_{k+1}^{-1}} + \| R_k x_k - R_{k+1}x_{k}\|_{U_{k+1}^{-1}} \nonumber \\
& \quad + \| R_{k+1} x_k - R_{k+1}x_{k+1} \|_{U_{k+1}^{-1}} + \lambda_k \|e_k  \|_{U_{k+1}^{-1}} \nonumber \\
& \leq (1-\lambda_k \alpha_k) \|x_k - R_{k}x_k \|_{U_{k+1}^{-1}} + \| R_k x_k - R_{k+1}x_{k}\|_{U_{k+1}^{-1}} + \|  x_k - x_{k+1} \|_{U_{k+1}^{-1}} \nonumber \\
& \quad  + \lambda_k \|e_k  \|_{U_{k+1}^{-1}} \nonumber \\
& \leq  \|x_k - R_{k}x_k \|_{U_{k+1}^{-1}} + \| R_k x_k - R_{k+1}x_{k}\|_{U_{k+1}^{-1}} + 2 \sqrt{\frac{1}{\alpha}}\lambda_k \|e_k  \|.
\end{align}
On the other hand, using the relation $R_k = (1-\frac{1}{\alpha_{k}})I +\frac{1}{\alpha_{k}} T_k$ and Lemma \ref{our-lemma4}, we have
\begin{align}\label{eq3-24}
& \quad \left\| R_k x_k - R_{k+1}x_{k}\right\|_{U_{k+1}^{-1}} \nonumber \\
& = \left\|   (1-\frac{1}{\alpha_{k}})x_k + \frac{1}{\alpha_{k}} T_k x_k - (1-\frac{1}{\alpha_{k+1}})x_k - \frac{1}{\alpha_{k+1}}T_{k+1}x_k       \right\|_{U_{k+1}^{-1}} \nonumber \\
& \leq  \left| \frac{1}{\alpha_{k+1}} - \frac{1}{\alpha_k} \right| \|x_k\|_{U_{k+1}^{-1}} + \left\| \frac{1}{\alpha_{k}} T_k x_k -  \frac{1}{\alpha_{k+1}}T_{k+1}x_k \right \|_{U_{k+1}^{-1}} \nonumber \\
& \leq  \left| \frac{1}{\alpha_{k+1}} - \frac{1}{\alpha_k} \right| \|x_k\|_{U_{k+1}^{-1}} + \left\| \frac{1}{\alpha_{k}} T_k x_k -  \frac{1}{\alpha_{k+1}}T_{k}x_k \right \|_{U_{k+1}^{-1}} \nonumber \\
& \quad + \left\| \frac{1}{\alpha_{k+1}} T_k x_k -  \frac{1}{\alpha_{k+1}}T_{k+1}x_k  \right\|_{U_{k+1}^{-1}} \nonumber \\
& \leq \frac{1}{2\beta} \left| \gamma_k \|U_k\| - \gamma_{k+1}\|U_{k+1}\| \right| \left(\|x_k\|_{U_{k+1}^{-1}} + \|T_k x_k\|_{U_{k+1}^{-1}} \right) + \frac{1}{\alpha_{k+1}}\|  T_k x_k - T_{k+1}x_k  \|_{U_{k+1}^{-1}} \nonumber \\
& \leq \frac{1}{2\beta} \left| \gamma_k \|U_k\| - \gamma_{k+1}\|U_{k+1}\| \right| \left(\|x_k\|_{U_{k+1}^{-1}} + \|T_k x_k\|_{U_{k+1}^{-1}} \right) + 2 \sqrt{\frac{1}{\alpha}} \|  T_k x_k - T_{k+1}x_k  \| \nonumber \\
& \leq \frac{1}{2\beta} \left| \gamma_k \|U_k\| - \gamma_{k+1}\|U_{k+1}\| \right| \left(\|x_k\|_{U_{k+1}^{-1}} + \|T_k x_k\|_{U_{k+1}^{-1}} \right) \nonumber \\
& \quad + 2 \sqrt{\frac{1}{\alpha}} \frac{\mu}{\gamma_{k+1}}  \left\| (\gamma_{k+1} U_{k}^{-1} - \gamma_k U_{k+1}^{-1})(x_k - T_{k+1}x_k)  \right\| \nonumber \\
& \leq \frac{1}{2\beta} \left| \gamma_k \|U_k\| - \gamma_{k+1}\|U_{k+1}\| \right| \left(\|x_k\|_{U_{k+1}^{-1}} + \|T_k x_k\|_{U_{k+1}^{-1}} \right) \nonumber \\
& \quad + 2 \sqrt{\frac{1}{\alpha}} \frac{\mu}{\underline{\gamma}} \left|\gamma_{k+1}-\gamma_{k} \right| \left\| U_{k}^{-1}(x_k - T_{k+1}x_k) \right \| \nonumber \\
& \quad + 2 \sqrt{\frac{1}{\alpha}} \frac{\mu}{\underline{\gamma}} \frac{2\beta}{\alpha} \left\| (U_{k}^{-1}-U_{k+1}^{-1})(x_k - T_{k+1}x_k) \right\|.
\end{align}
The combination of (\ref{eq3-24}) with (\ref{eq3-23}) yields
\begin{align}\label{eq3-25}
& \quad \|x_{k+1} - R_{k+1}x_{k+1}\|_{U_{k+1}^{-1}} \nonumber \\
& \leq \|x_k - R_{k}x_k \|_{U_{k+1}^{-1}} + \frac{1}{2\beta} \left| \gamma_k \|U_k\| - \gamma_{k+1}\|U_{k+1}\| \right | \left(\|x_k\|_{U_{k+1}^{-1}} + \|T_k x_k\|_{U_{k+1}^{-1}} \right) \nonumber \\
& \quad + 2 \sqrt{\frac{1}{\alpha}} \frac{\mu}{\underline{\gamma}} |\gamma_{k+1}-\gamma_{k}| \left\| U_{k}^{-1}(x_k - T_{k+1}x_k) \right \| \nonumber \\
& \quad + 2 \sqrt{\frac{1}{\alpha}} \frac{\mu}{\underline{\gamma}} \frac{2\beta}{\alpha} \left\| (U_{k}^{-1}-U_{k+1}^{-1})(x_k - T_{k+1}x_k) \right \| + 2 \sqrt{\frac{1}{\alpha}}\lambda_k \|e_k  \| \nonumber \\
& \leq (1+\eta_k)\|x_k - R_{k}x_k \|_{U_{k}^{-1}} + \frac{1}{2\beta} \left | \gamma_k \|U_k\| - \gamma_{k+1}\|U_{k+1}\| \right | \left(\|x_k\|_{U_{k+1}^{-1}} + \|T_k x_k\|_{U_{k+1}^{-1}}\right ) \nonumber \\
& \quad + 2 \sqrt{\frac{1}{\alpha}} \frac{\mu}{\underline{\gamma}} |\gamma_{k+1}-\gamma_{k}| \left\| U_{k}^{-1}(x_k - T_{k+1}x_k)  \right\| \nonumber \\
& \quad + 2 \sqrt{\frac{1}{\alpha}} \frac{\mu}{\underline{\gamma}} \frac{2\beta}{\alpha} \left\| (U_{k}^{-1}-U_{k+1}^{-1})(x_k - T_{k+1}x_k) \right\| + 2 \sqrt{\frac{1}{\alpha}}\lambda_k \|e_k\|.
\end{align}
 With the help of Lemma \ref{conergence-lemma1}, we can conclude from (\ref{eq3-25}) that $\lim_{k\rightarrow +\infty}\|x_k - R_k x_k\|_{U_{k}^{-1}}=0$. Hence, $\lim_{k\rightarrow +\infty}\|x_k - R_k x_k\|=0$. As a consequence, $\lim_{k\rightarrow +\infty}\|x_k - T_k x_k\|=0$.

(iii) and (iv) can be proven using the same proof as Theorem \ref{main-theorem}.

\end{proof}

\begin{remark}
In Theorem \ref{main-theorem2}, we obtain the weak convergence of the iterative sequence generated by (\ref{variable-forward-backward2}) with a weaker condition on $\{\lambda_k\}$ than Theorem \ref{main-theorem}.

\end{remark}

In Theorems \ref{main-theorem} and \ref{main-theorem2}, let $ U_{k}=I $, in which case we obtain the following corollary, which shows the convergence of the forward-backward splitting algorithm with variable step sizes.

\begin{corollary}\label{corollary}
Let $H$ be a real Hilbert space. Let $A:H\rightarrow 2^{H} $ be maximal monotone. Let $B:H\rightarrow H $ be $\beta$-cocoercive,
for some $\beta > 0$. Suppose that $\Omega = \textrm{zer} (A+B) \neq \emptyset$.
Let $\{\gamma_{k}\} \subset (0, 2\beta)$, and $\{\lambda_k\} \subset (0,\frac{1}{\alpha_{k}})$, where $\alpha_{k}=\frac{2 \beta}{4 \beta-\gamma_{k}}$.
Let $\{a_{k}\}$ and $\{b_{k}\}$ be two sequences in $H$ such that $\sum_{k=0}^{+\infty}\lambda_k \|a_k\|<+\infty$ and $\sum_{k=0}^{+\infty}\lambda_k \|b_k\|< +\infty$. Let $x_{0}\in H$, and set
\begin{equation}\label{relax-forward-backward-fixed-metric}
\left\{
\begin{aligned}
y_{k}   & =  x_{k}-\gamma_{k}(Bx_{k}+b_{k}), \\
x_{k+1} & =  x_{k}+\lambda_{k}(J_ {\gamma_{k}A}(y_{k})+a_{k}-x_{k}).
\end{aligned}
\right.
\end{equation}
Then, we have

\noindent \emph{(i)} for any $x^{*}\in \Omega$, $\lim_{k\rightarrow +\infty}\|x_{k}-x^{*}\|$ exists;

Suppose that

\noindent (a1) $0< \underline {\lambda}\leq \lambda_{k}$;

\noindent (a2) $\lambda_{k} \leq \frac{1}{\alpha_{k}}-\tau$, where $\tau \in (0,\frac{1}{\alpha_{k}}-\underline{\lambda})$;

\noindent (a3) $0< \underline{\gamma} \leq \gamma_k $;

\noindent (a4) $\gamma_k \leq 2\beta-\epsilon $, where $\epsilon \in (0,2\beta- \underline{\gamma})$.

\noindent (a5)\ $\sum_{k=0}^{+\infty}\lambda_k (\frac{1}{\alpha_k}-\lambda_k) = +\infty$ and $\sum_{k=0}^{+\infty}|  \gamma_{k+1} - \gamma_{k} | < +\infty$.

If the conditions of (a1)-(a2) or (a3)-(a5) hold, then we have

\noindent \emph{(ii)} $\lim_{k\rightarrow +\infty} \| x_{k} - J_{\gamma_k A}(x_k - \gamma_k Bx_k) \| =0 $;

If the conditions of (a1)-(a3) or (a3)-(a5) hold, then we have

\noindent \emph{(iii)}  $\{x_{k}\}$ converges weakly to a point in $\Omega$;

If the conditions of (a1)-(a3) or (a1), (a3)-(a5) hold, then we have

\noindent \emph{(iv)} $Bx_k \rightarrow Bx^{*}$ as $k\rightarrow +\infty$, where $x^{*}\in \Omega$.

\end{corollary}

\begin{remark}
Under the condition (a1)-(a3), Corollary \ref{corollary} reaffirms Proposition 4.4 of Combettes and Yamada \cite{Combettes2015JMAA}. In addition, we obtain the convergence of the iterative scheme (\ref{relax-forward-backward-fixed-metric}) under the condition (a3)-(a5), which provides a weaker assumption on the relaxation parameters $\lambda_k$ than the condition (a1) and (a2). Consequently, the obtained results improve and generalize Proposition 4.4 of Combettes and Yamada \cite{Combettes2015JMAA}.

\end{remark}

As an application of Theorems \ref{main-theorem} and \ref{main-theorem2}, we can obtain the following convergence results for solving convex minimization problem (\ref{sum-two-convex}).

\begin{corollary}\label{coro2}
Let $H$ be a real Hilbert space. Let $g:H\rightarrow (-\infty,+\infty]$ be a proper, lower semi-continuous, convex function. Let $f:H\rightarrow R$ be convex and differentiable with a $1/\beta$-Lipschitz continuous gradient. Assume that $\Omega$ is the set of solutions of problem (\ref{sum-two-convex}) and  $\Omega \neq \emptyset$.  Let $x_{0}\in H$, and set
\begin{equation}\label{forward-backward-convex-optimization}
\left\{
\begin{aligned}
y_{k}   & =  x_{k}-\gamma_{k}U_{k}(\nabla f (x_{k})+b_{k}), \\
x_{k+1} & =  x_{k}+\lambda_{k}(prox_{\gamma_{k}g}^{U_{k}^{-1}}(y_{k})+a_{k}-x_{k}),
\end{aligned}
\right.
\end{equation}
where $\{U_k\}$, $\{\gamma_{k}\}$, $\{\lambda_k\}$, $\{a_{k}\}$, and $\{b_{k}\}$ satisfy the same conditions as in Theorem \ref{main-theorem} or Theorem \ref{main-theorem2}.

Then the following hold:

\noindent (i) For any $x^{*}\in \Omega$, $\lim_{k\rightarrow +\infty}\|x_{k}-x^{*}\|_{U_{k}^{-1}}$ exists;

\noindent (ii) $\lim_{k\rightarrow +\infty} \| x_{k} - prox_{\gamma_k g}^{U_{k}^{-1}}(x_k - \gamma_k U_k \nabla f(x_k)) \| =0 $;

\noindent (iii) $\{x_{k}\}$ converges weakly to a point in $\Omega$;

\noindent (iv) $\nabla f(x_k) \rightarrow \nabla f(x^{*})$ as $k\rightarrow +\infty$, where $x^{*}\in \Omega$.
\end{corollary}

\begin{proof}
Because $f$ is convex differentiable, according to the Baillon-Haddad theorem, $\nabla f $ is $\beta$-cocoercive.
From the definition of the proximity operator on the Hilbert space $H_{U^{-1}}$, we know that
\begin{equation}
prox_{\gamma_{k}g}^{U_{k}^{-1}}(u)=J_ {\gamma_{k}U_{k}\partial g}(u).
\end{equation}
Set $A=\partial g$ and $B=\nabla f$ in Theorem \ref{main-theorem} or Theorem \ref{main-theorem2} and this enables us to confirm the conclusions of Corollary \ref{coro2}.
\end{proof}

\section{Applications}

In this section, we present our study of several applications of the variable metric forward-backward splitting algorithm.

\subsection{Application to variational inequality problem}

Consider the following variational inequality problem (VIP):
\begin{equation}\label{vip}
\textrm{find }\ x^* \in C, \quad \textrm{such that }\ \langle Bx^*, y-x^*  \rangle \geq 0, \quad \forall y\in C,
\end{equation}
where $C$ is a nonempty closed convex subset of $H$, and $B:H\rightarrow H$ is a nonlinear operator.

Recall the indicator function $\delta_{C}$, which is defined as
\begin{equation}
\delta_{C}(x)=\left\{
\begin{aligned}
& 0,  & x\in C \\
& +\infty,  &\textrm{otherwise}. \\
\end{aligned}
\right.
\end{equation}
The proximal operator of $\delta_{C}$ is well known to be the metric projection on $C$, which is defined by
$$
P_{C}(x) = prox_{\delta_{C}}(x) = \arg\min_{y\in C} \|x-y\|.
$$

The normal cone operator of $C$ is $N_{C}$, which is defined by
\begin{equation}
N_{C}(x)=\left\{
\begin{aligned}
& \{ w| \langle w, y-x \rangle \leq 0, \forall y\in C \},  & x\in C \\
& \emptyset,  &\textrm{otherwise}. \\
\end{aligned}
\right.
\end{equation}
Then, VIP (\ref{vip}) is equivalent to the following monotone inclusion problem:
\begin{equation}\label{vip-monotone}
0\in Bx + N_{C}(x).
\end{equation}
Assuming that $B$ is $\beta$-cocoercive, then (\ref{vip-monotone}) is a special case of the monotone inclusion problem (\ref{two-monotone-inclusion}). Let $A=N_{C}$, then we know that $J_{\gamma U A} = P_{C}^{U^{-1}}$, for any $\gamma >0$ and $U\in \mathcal{P}_{\alpha}(H)$. The operator $P_{C}^{U^{-1}}$ denotes the projector onto a nonempty closed convex subset $C$ of $H$ relative to the norm $\|\cdot\|_{U^{-1}}$. More precisely,
$$
P_{C}^{U^{-1}}(x) =  \arg\min_{y\in C} \|x-y\|_{U^{-1}}.
$$
On the basis of Theorems \ref{main-theorem} and \ref{main-theorem2}, we obtain the following convergence theorem to solve the VIP (\ref{vip}).

\begin{theorem}
Let $H$ be a real Hilbert space. Let $B:H\rightarrow H$ be a $\beta$-cocoercive operator. We denote by $\Omega$ the solution set of VIP (\ref{vip}) and assume that $\Omega \neq \emptyset$.  Let $x_{0}\in H$, set
\begin{equation}\label{vip-iterative-algorithm}
\left\{
\begin{aligned}
y_{k}   & =  x_{k}-\gamma_{k}U_{k}(Bx_{k}+b_{k}), \\
x_{k+1} & =  x_{k}+\lambda_{k}(P_{C}^{U_{k}^{-1}}(y_{k})+a_{k}-x_{k}),
\end{aligned}
\right.
\end{equation}
where $\{U_k\}$, $\{\gamma_{k}\}$, $\{\lambda_k\}$, $\{a_{k}\}$, and $\{b_{k}\}$ satisfy the same conditions as in Theorem \ref{main-theorem} or Theorem \ref{main-theorem2}.

Then the following hold:

\noindent (i) For any $x^{*}\in \Omega$, $\lim_{k\rightarrow +\infty}\|x_{k}-x^{*}\|_{U_{k}^{-1}}$ exists;

\noindent (ii) $\lim_{k\rightarrow +\infty} \| x_{k} - P_{C}^{U_{k}^{-1}}(x_k - \gamma_k U_k Ax_k) \| =0 $;

\noindent (iii) $\{x_{k}\}$ converges weakly to a point in $\Omega$;

\noindent (iv) $Bx_k \rightarrow Bx^{*}$ as $k\rightarrow +\infty$, where $x^{*}\in \Omega$.

\end{theorem}

\subsection{Application to constrained convex minimization problem}

Consider the following constrained convex minimization problem:
\begin{equation}\label{constrained-optim}
\begin{aligned}
\min\ & f(x) \\
s.t. \ & x\in C,
\end{aligned}
\end{equation}
where $C$ is a nonempty closed convex subset of $H$, and $f:H\rightarrow R$ is a proper closed convex differentiable function with a Lipschitz continuous gradient.

It follows from the definition of the indicator function that constrained convex minimization problem (\ref{constrained-optim}) is equivalent to the following unconstrained minimization problem:
\begin{equation}\label{unconstrained-optim}
\min_{x\in H}\ f(x) + \delta_{C}(x).
\end{equation}
It is obvious that problem (\ref{unconstrained-optim}) is a special case of (\ref{sum-two-convex}). Therefore, by taking $g(x) = \delta_{C}(x)$, we obtain the following convergence theorem for solving constrained convex minimization problem (\ref{constrained-optim}).

\begin{theorem}
Let $H$ be a real Hilbert space. Let $f:H\rightarrow R$ be a proper, closed convex function such that $f$ is differentiable with an $L$-Lipschitz continuous gradient. We denote by $\Omega$ the solution set of the constrained convex minimization problem (\ref{vip}) and assume that $\Omega \neq \emptyset$.  Let $x_{0}\in H$, and set
\begin{equation}\label{constrained-iterative-algorithm}
\left\{
\begin{aligned}
y_{k}   & =  x_{k}-\gamma_{k}U_{k}(\nabla f (x_{k})+b_{k}), \\
x_{k+1} & =  x_{k}+\lambda_{k}(P_{C}^{U_{k}^{-1}}(y_{k})+a_{k}-x_{k}),
\end{aligned}
\right.
\end{equation}
where $\{U_k\}$, $\{\gamma_{k}\}$, $\{\lambda_k\}$, $\{a_{k}\}$, and $\{b_{k}\}$ satisfy the same conditions as in Theorem \ref{main-theorem} or Theorem \ref{main-theorem2}.

Then the following hold:

\noindent (i) For any $x^{*}\in \Omega$, $\lim_{k\rightarrow +\infty}\|x_{k}-x^{*}\|_{U_{k}^{-1}}$ exists;

\noindent (ii) $\lim_{k\rightarrow +\infty} \| x_{k} - P_{C}^{U_{k}^{-1}}(x_k - \gamma_k U_k \nabla f(x_k)) \| =0 $;

\noindent (iii) $\{x_{k}\}$ converges weakly to a point in $\Omega$;

\noindent (iv) $\nabla f(x_k) \rightarrow \nabla f(x^{*})$ as $k\rightarrow +\infty$, where $x^{*}\in \Omega$.

\end{theorem}

\subsection{Application to split feasibility problem}

Consider the split feasibility problem (SFP) as follows:
\begin{equation}\label{sfp}
\textrm{find } x\in C, \quad \textrm{such that } Lx \in Q,
\end{equation}
where $C$ and $Q$ are nonempty, closed convex subsets of Hilbert spaces $H$ and $G$, respectively. $L:H \rightarrow G$ is a bounded linear operator. SFP (\ref{sfp}) was first introduced by Censor and Elfving \cite{censorandelfving1994} in a finite dimensional Hilbert space and has since been extensively studied by many authors, see, for example \cite{xuhk2006,xuhk2010} and references therein.

SFP (\ref{sfp}) is closely related to constrained convex minimization problem (\ref{constrained-optim}). More precisely, the corresponding constrained convex minimization problem of SFP (\ref{sfp}) is,
\begin{equation}\label{unconstrained-sfp}
\begin{aligned}
\min_{x}\ & \frac{1}{2}\|  x-P_{Q}(Lx) \|^2 \\
s.t. \ & x\in C.
\end{aligned}
\end{equation}
Let $x^*$ be a solution of SFP (\ref{sfp}), then $x^*$ is a solution of (\ref{unconstrained-sfp}). Conversely, let $x^*$ be a solution of (\ref{unconstrained-sfp}) and $f(x):=\frac{1}{2}\|  x-P_{Q}(Lx) \|^2 = 0$, then $x^*$ is a solution of SFP (\ref{sfp}). Under the assumption that the solution set of SFP (\ref{sfp}) is nonempty, SFP (\ref{sfp}) and constrained convex minimization problem (\ref{unconstrained-sfp}) are equivalent.

The function $f(x) = \frac{1}{2}\|  x-P_{Q}(Lx) \|^2$ is convex differentiable and the gradient operator $\nabla f(x) = L^{*}(Lx - P_{Q}(Lx))$ is $\frac{1}{\|L\|^2}$-cocoercive. Therefore, we obtain the following theorem for solving SFP (\ref{sfp}).

\begin{theorem}
Let $H$ and $G$ be real Hilbert spaces. Let $L:H\rightarrow G$ be a bounded linear operator. Let $C$ and $Q$ be nonempty closed and convex subsets of $H$ and $G$, respectively. We denote by $\Omega$ the solution set of SFP (\ref{sfp}) and assume that $\Omega \neq \emptyset$.  Let $x_0\in H$, and set
\begin{equation}\label{sfp-iterative-algorithm}
\left\{
\begin{aligned}
y_{k}   & =  x_{k}-\gamma_{k}U_{k}(L^{*}(Lx_k - P_{Q}(Lx_{k}))+b_{k}), \\
x_{k+1} & =  x_{k}+\lambda_{k}(P_{C}^{U_{k}^{-1}}(y_{k})+a_{k}-x_{k}),
\end{aligned}
\right.
\end{equation}
where $\{U_k\}$, $\{\gamma_{k}\}$, $\{\lambda_k\}$, $\{a_{k}\}$, and $\{b_{k}\}$ satisfy the same conditions as in Theorem \ref{main-theorem} or Theorem \ref{main-theorem2}.

Then the following hold:

\noindent (i) For any $x^{*}\in \Omega$, $\lim_{k\rightarrow +\infty}\|x_{k}-x^{*}\|_{U_{k}^{-1}}$ exists;

\noindent (ii) $\lim_{k\rightarrow +\infty} \| x_{k} - P_{C}^{U_{k}^{-1}}(x_k - \gamma_k U_k L^{*}(Lx_k - P_{Q}(Lx_{k}))) \| =0 $;

\noindent (iii) $\{x_{k}\}$ converges weakly to a point in $\Omega$;

\noindent (iv) $L^{*}(Lx_k - P_{Q}(Lx_{k})) \rightarrow L^{*}(Lx^{*} - P_{Q}(Lx^{*}))$ as $k\rightarrow +\infty$, where $x^{*}\in \Omega$.

\end{theorem}

\begin{remark}
To the best of our knowledge, the proposed iterative algorithms (\ref{vip-iterative-algorithm}), (\ref{constrained-iterative-algorithm}), and (\ref{sfp-iterative-algorithm}) are the most general ones for solving variational inequality problem (\ref{vip}), constrained convex minimization problem (\ref{constrained-optim}), and split feasibility problem (\ref{sfp}), respectively. Most of the existing algorithms  \cite{Byrne2004,yangandzhao2006,xuhk2010,xuhk2011,Wangfh2014JAM} are special cases of ours.

\end{remark}

\section{Numerical experiments}

In this section, we apply the proposed iterative algorithm (\ref{forward-backward-convex-optimization}) to solve the famous LASSO problem \cite{Tibshirani1996}. All the experiments are performed on a standard Lenovo Laptop with Intel (R) Core (TM) i7-4712MQ 2.3 GHZ CPU and 4 GB RAM. We run the program with MATLAB 2014a.

Let's recall the LASSO problem:
\begin{equation}\label{lasso}
\begin{aligned}
\min_{x\in R^{n}}\, & \frac{1}{2}\|Ax-b\|_{2}^{2} \\
s.t.\, & \|x\|_1 \leq t,
\end{aligned}
\end{equation}
where $A\in R^{m\times n}$, $b\in R^{m}$ and $t>0$. Define $C:=\{ x | \|x\|_1 \leq t \}$, by using the indicator function, we see that (\ref{lasso}) is equivalent to the following unconstrained optimization problem
\begin{equation}\label{lasso-unconstrained}
\min_{x}\, \frac{1}{2}\|Ax-b\|_{2}^{2} + \delta_{C}(x),
\end{equation}
which is a special case of the general optimization problem (\ref{sum-two-convex}). Let $f(x) = \frac{1}{2}\|Ax-b\|_{2}^{2}$ and $g(x) = \delta_{C}(x)$, then we can apply iterative algorithm (\ref{forward-backward-convex-optimization}) to solve (\ref{lasso-unconstrained}). Notice that the gradient of $f(x)$ is $\nabla f(x) = A^{T}(Ax-b)$ and the Lipschitz constant of $\nabla f$ is $L:=\|A\|^2$. Besides, the proximity operator of indicator function $\delta_{C}(x)$ is the orthogonal projection onto the closed convex set $C$. Although it has no closed-form solution, it can be calculated in a polynomial time.

In the tests, the true signal $x\in R^{n}$ has $k$ non-zero elements, which is generated from uniform distribution in the interval $[-2,2]$. The system matrix $A\in R^{m\times n}$ is generated from standard Gaussian distribution. The observed signal $b$ is given by $b=Ax$. In the experiment, we set $m=240$, $n=1024$ and $k=40$.  The stopping criterion is defined as,
\begin{equation}
\frac{\|x_{k+1}-x_k\|_2}{\|x_k\|_2} \leq \varepsilon,
\end{equation}
where $\varepsilon >0$ is a small constant. We test the performance of the proposed iterative algorithm with different choices of the step size $\gamma_k$ and the relaxation parameter $\lambda_k$. For simplicity, we set them as constant during the iteration process. According to Corollary \ref{coro2}, we know that $\gamma_k \in (0,\frac{2}{L})$ and $\lambda_k \in (0,\frac{4-\gamma_k L}{2})$. The obtained numerical results are listed in Table \ref{results}, in which we report the number of iterations (``$Iter$"), the objective function value (``$Obj$") and the error between the recovered signal and the true signal (``$Err$"). We can see from Table \ref{results} that when the step size $\gamma_k$ is fixed, a large relaxation parameter $\lambda_k$ leads to a faster convergence. At the same time, the larger the step size, the faster the algorithm converges.

\begin{table}[h]
 \scriptsize
\centering
\caption{Numerical results for different choices of $\gamma_k$ and $\lambda_k$ for solving the LASSO problem (\ref{lasso})}
\begin{tabular}{c|c|ccccccc}
\hline
\multirow{2}[1]{*}{$\gamma_k$} & \multirow{2}[1]{*}{$\lambda_k$} &  \multicolumn{3}{c}{$\varepsilon = 10^{-6}$} &  & \multicolumn{3}{c}{$\varepsilon = 10^{-8}$}  \\ \cline{3-5} \cline{7-9}
 &  &   $Iter$ & $Err$ & $Obj$ & & $Iter$  & $Err$ & $Obj$ \\
\hline
\hline
\multirow{8}[1]{*}{$\frac{1}{2L}$} & $0.2$  &  $16553$ & $0.0246$ & $0.0020$ & &  $32336$ & $2.4764e-4$ & $1.9778e-7$   \\
 &  $0.4$  &   $9457$ & $0.0123$ & $4.9063e-4$ & &  $17357$ & $1.2383e-4$ & $4.9432e-8$   \\
&  $0.6$  &   $6765$ & $0.0082$ & $2.1851e-4$ & &  $12035$ & $8.2498e-5$ & $2.1935e-8$   \\
&  $0.8$  &   $5319$ & $0.0062$ & $1.2296e-4$ & &  $9272$ & $6.1880e-5$ & $1.2339e-8$  \\
&  $1$  &   $4408$ & $0.0049$ & $7.8535e-5$ & &  $7570$ & $4.9492e-5$ & $7.8918e-9$   \\
&  $1.2$  &   $3777$ & $0.0041$ & $5.4500e-5$ & &  $6412$ & $4.1228e-5$ & $5.4756e-9$ \\
&  $1.5$  &   $3123$ & $0.0033$ & $3.4835e-5$ & &  $5231$ & $3.2952e-5$ & $3.4975e-9$  \\
&  $1.75$  &   $2736$ & $0.0028$ & $2.5671e-5$ & &  $4543$ & $2.8267e-5$ & $2.5735e-9$  \\
\hline
\multirow{7}[1]{*}{$\frac{1}{L}$} & $0.2$  &  $9455$ & $0.0123$ & $4.9106e-4$ & &  $17356$ & $1.2381e-4$ & $4.9417e-8$   \\
 &  $0.4$  &   $5319$ & $0.0062$ & $1.2280e-4$ & &  $9271$ & $6.1913e-5$ & $1.2352e-8$   \\
&  $0.6$  &   $3776$ & $0.0041$ & $5.4633e-5$ & &  $6412$ & $4.1206e-5$ & $5.4698e-9$   \\
&  $0.8$  &   $2955$ & $0.0031$ & $3.0621e-5$ & &  $4931$ & $3.0909e-5$ & $3.0772e-9$  \\
&  $1$  &   $2440$ & $0.0025$ & $1.9556e-5$ & &  $4021$ & $2.4670e-5$ & $1.9601e-9$   \\
&  $1.2$  &   $2085$ & $0.0020$ & $1.3549e-5$ & &  $3402$ & $2.0554e-5$ & $1.3605e-9$ \\
&  $1.5$  &   $1718$ & $0.0016$ & $8.6760e-6$ & &  $2771$ & $1.6469e-5$ & $8.7329e-10$  \\
\hline
\multirow{6}[1]{*}{$\frac{1.9}{L}$} & $0.2$  &  $5550$ & $0.0065$ & $1.3624e-4$ & &  $9711$ & $6.5144e-5$ & $1.3675e-8$   \\
 &  $0.4$  &   $3086$ & $0.0032$ & $3.4008e-5$ & &  $5167$ & $3.2508e-5$ & $3.4038e-9$   \\
&  $0.6$  &   $2178$ & $0.0022$ & $1.5100e-5$ & &  $3565$ & $2.1657e-5$ & $1.5104e-9$   \\
&  $0.8$  &   $1698$ & $0.0016$ & $8.4390e-6$ & &  $2737$ & $1.6253e-5$ & $8.5059e-10$  \\
&  $1$  &   $1398$ & $0.0013$ & $5.3842e-6$ & &  $2229$ & $1.2973e-5$ & $5.4181e-10$   \\
&  $1.05$  &   $1340$ & $0.0012$ & $4.8598e-6$ & &  $2131$ & $1.2350e-5$ & $4.9099e-10$ \\
\hline
\end{tabular}\label{results}

\end{table}

In order to more visualize the effect of iterative parameters on the value of the function, Figure \ref{objective-function-value} shows the objective function value against the number of iterations. Further, we plot the true signal and the recovered signal in Figure \ref{recovered-signal} for the parameters of $\gamma_k = \frac{1.9}{L}$, $\lambda_k = 1.05$ and the stopping criterion $\varepsilon = 10^{-8}$. We can see from Figure \ref{recovered-signal} that the true signal is successfully reconstructed.

\begin{figure}[htbp]
     \setlength{\abovecaptionskip}{-5pt}
  \centering \scalebox{0.7}{\includegraphics{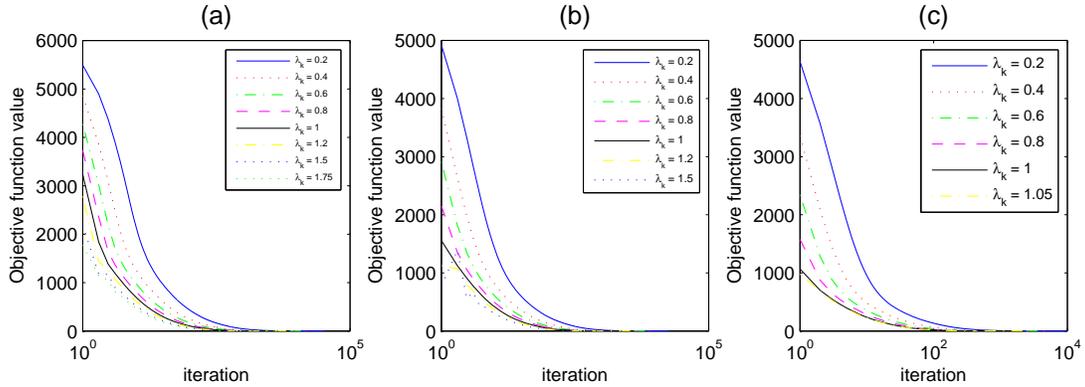}}
  \caption[]{The objective function value against the number of iterations for the LASSO problem. (a) $\gamma_k = \frac{1}{2L}$, (b) $\gamma_k = \frac{1}{L}$ and (c) $\gamma_k = \frac{1.9}{L}$.}%
{\label{objective-function-value}}
\end{figure}

\begin{figure}[htbp]
     \setlength{\abovecaptionskip}{-5pt}
  \centering \scalebox{0.65}{\includegraphics{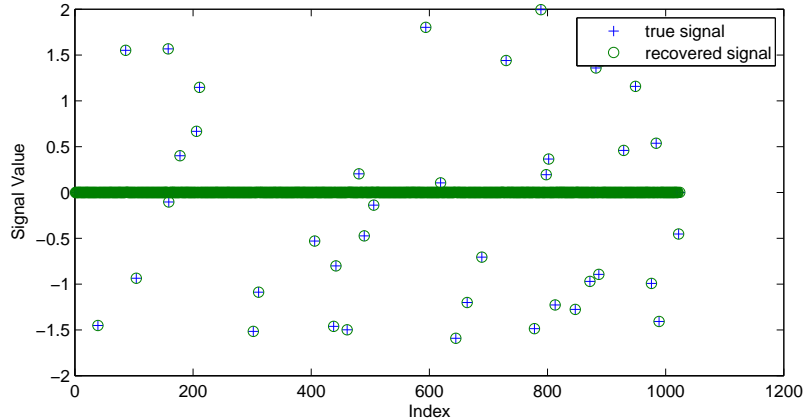}}
  \caption[]{The recovered sparse signal versus the true $k$-sparse signal.}%
{\label{recovered-signal}}
\end{figure}

\section{Conclusions}

In this paper, we proposed a new convergence analysis of the variable metric forward-backward splitting algorithm (\ref{variable-forward-backward}) with extended relaxation parameters. Based on the averaged operator $J_{\gamma_k U_k A}(I-\gamma_k U_k B)$ and the firmly nonexpansive $J_{\gamma_k U_k A}$ on the Hilbert spaces $H_{U_{k}^{-1}}$, we proved the weak convergence of this algorithm. Compared to existing work, we imposed a slightly weak condition on the relaxation parameters to ensure the convergence of the forward-backward splitting algorithm when using the variable metric and variable step sizes. Our
results complemented and extended the corresponding results of Combettes and Yamada \cite{Combettes2015JMAA}. Furthermore, we obtained several general iterative algorithms for solving the variational inequality problem, the constrained convex minimization problem, and the split feasibility problem, respectively. These results generalized and improved the known results in the literature. Numerical experimental results on LASSO problem showed that the step size $\gamma_k$ and relaxation parameter $\lambda_k$ had much impact on the convergence speed of the proposed iterative algorithm. The larger the step size, the faster the algorithm converged. The over-relaxation parameter $\lambda_k$ ($\lambda_k >1$) performed better than the under-relaxation parameter $\lambda_k$ ($\lambda_k \leq 1$).

\section*{Acknowledgement}
This work was supported by the National Natural Science Foundations
of China (11661056, 11771198, 11401293), the Postdoctoral Research Foundation of China (2015M571989)
and the Postdoctoral Science Foundation of Jiangxi Province (2015KY51).

\section*{Conflict of interest}
The authors declare no conflict of interest.


\end{document}